\documentclass{article}

\usepackage{arxiv}

\usepackage[utf8]{inputenc} 
\usepackage[T1]{fontenc}    
\usepackage{hyperref}       
\usepackage{url}            
\usepackage{booktabs}       
\usepackage{amsfonts}       
\usepackage{nicefrac}       
\usepackage{microtype}      
\usepackage{lipsum}
\usepackage{color}
\usepackage[usenames,dvipsnames,svgnames,table]{xcolor}
\usepackage{graphicx}
\usepackage{mwe}
\usepackage{graphbox}
\usepackage{amssymb}
\usepackage{extarrows}
\usepackage{tikz}
\usetikzlibrary{matrix}
\usepackage{amsthm}
\usepackage{tikz-cd}
\usepackage{overpic}

\begin{document}

\begin{center}
\vspace{0.6cm}

{ \bf \LARGE   The Effect of Singularization on the Euler Characteristic}

\vspace{0.6cm}
\large 
M.A. de Jesus Zigart \footnote{Partially  financed  by the Coordenação de Aperfeiçoamento de Pessoal de Nível Superior – Brasil (CAPES) – Finance Code 001}
\hspace{0.3cm}
K.A. de Rezende\footnote{Partially supported by CNPq under grant 305649/2018-3 and FAPESP under grant 2018/13481-0.}
\hspace{0.3cm}
N.G. Grulha Jr. \footnote{Supported by CNPq under grant 303046/2016-3; and  FAPESP under grant 2017/09620-2.}
\hspace{0.3cm}
D.V.S. Lima
\hspace{0.3cm}

\end{center}

\vspace{0.5cm}

\begin{abstract}

In this work, singular surfaces are obtained from smooth orientable closed surfaces by applying three basic simple loop operations,  \textit{collapsing operation}, \textit{zipping  operation} and \textit{double loop identification}, each of which produces different singular surfaces.  A formula that provides the Euler characteristic of the singularized surface is proved.  Also, we introduce a new definition of genus for singularized surfaces which generalizes the classical definition of genus in the smooth case. A theorem relating the Euler characteristic to the genus of the singularized surface is proved.
\end{abstract}

\vspace{0.5cm}


\vspace{0.25cm}

\section{Introduction}

Singularities  appear in several fields of study as a sign of qualitative change. We may experience them in Calculus, representing maximum or minimum points of a function; in Dynamical Systems, as stationary solutions that characterize the behaviour of solutions in their vicinity; or in Physics, where they can appear on larger scales, for instance, when a massive star undergoes a gravitational collapse after exhausting its internal nuclear fuel, which can lead to the birth of black holes or naked singularities, the latter being discussed as potential particle accelerators, acting like cosmic super-colliders \cite{patil2010naked,patil2011naked}. The formation of these so called spacetime singularities is a more general phenomena in which general theory of relativity plays an important role \cite{joshi2014spacetime}. And the most appealing example of such singularity is perhaps the Big Bang.  

Singularity Theory, structured as a field of research, has risen from the work of \textcolor{black}{Hassler Whitney, John Mather and René Thom. } It is the field of science dedicated to studying singularities in their many occurances. One approach is to consider the embedding of a $m$-dimensional smooth manifold in an Euclidean space of dimension lower than $2m$. According to Whitney's embedding theorem \cite{whitney1944self}, by doing so, the embedding will necessarily cause the manifold to self-intersect, originating singular sets. An interesting question is to consider the types of singularities produced in this fashion and to study their stability. 

In addition, understanding how smooth manifolds come together at their singular points is another concern in Singularity Theory, studied oftenly under resolutions of singular algebraic surfaces, that is, surfaces given as the zeroset of one polynomial in three variables. Resolving a singular surface means trying to find a surjective map from a smooth surface to the singular surface, which is an isomorphism almost everywhere.

Finding blowups, which make up the resolution map, can be highly nontrivial \cite{faber2010today}. However, one can perform the inverse process, called a blowdown,   to produce singularities on a smooth surface   by contracting a hypersurface. It is clear that this construction yields the resolution of the singularity obtained this way. But the singularity produced by an arbitrary blowdown may not be singularity one hopes for.

This interplay between blowdowns and blowups can be depicted, for instance, by a continuous deformation of a torus onto a pinched torus. In this scenario, both maps are given intuitively by the initial and final stage of this deformation, as one goes back  and forth  in time, blowup and blowdown, respectively. 
The deformation gives rise to a family of surfaces called a smoothing of the singular surface, 
 which in this case is produced by a vanishing cycle.

\begin{figure}[!h] 
\centering

  \begin{tikzcd} \hspace{0.2cm}
  \begin{overpic}[align=c,scale=0.5]{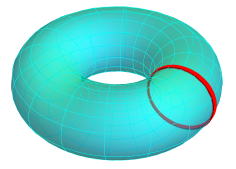}
  \end{overpic} \hspace{1cm}
    \begin{overpic}[align=c,scale=0.5]{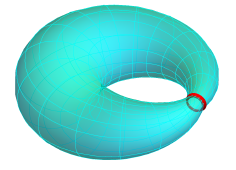}
    \end{overpic} \hspace{1cm}
    \begin{overpic}[align=c,scale=0.5]{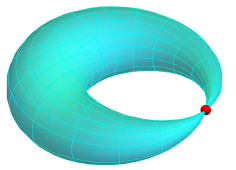}
    \end{overpic}
  \end{tikzcd}

\caption{Continuous deformation of the torus (left-most) onto the pinched torus (right-most)}
\end{figure}

The study of deformations, smoothings, vanishing cycles, unfoldings and bifucartions are oftenly posed, along with resolutions, in order to understand the topology of singular spaces. 
 Brasselet \textcolor{black}{ \textit{et al.}}, in \cite{brasselet2009vector}, propose the study by vector field methods, which is very useful to compute the Euler characteristic of such surfaces. For the singular surfaces given as the image of stable mappings, a formula that computes their Euler characteristic was proved in \cite{izumiya1995topologically} by Izumiya and Marar.

The Euler characteristic is a topological invariant that fascinates by its simplicity and yet many applications. It is part of a family of morphological measures, called Minkowski functionals,
 and represents a very compact way of characterizing the connectivity of complex image structures, since it is invariant  under deformation or scaling. Moreover, its additivity makes it possible to extend the practical application of Minkowski functionals beyond convex bodies. They've been applied in the topological study of the density distribution of galaxies in astrophysics \cite{schmalzing1995minkowski,schmalzing1997beyond}, to improve the diagnosis of osteoporosis \cite{rath2008strength}, X-ray analysis in digital mammography \cite{boehm2008automated} and to enhance the accuracy of brain tumor classification \cite{huml2013brain}. 

In this paper, we consider quotient maps inspired on blowdowns to produce singular surfaces from a family of smooth closed ones. In this more abstract setting, we do not need the concept of stable mappings, but we can say that our operations were inspired by the study of stable singularities. The types of singularities obtained here are: cones, cross-caps and double crossings. And then, we study the resulting singular surfaces under the light of their Euler characteristic.

\section{Preliminaries}

\textcolor{black}{Certainly among the most fruitful and beautiful formulas in the history of Mathematics}, is the Euler formula for a  polyhedron $P$ given by
 $\mathcal{X} (P)= V - E + F=2 $, where $V$ is the number of vertices, $E$ the number of edges and $F$ the number of faces. See \cite{richeson2019euler}.
This beloved formula has entertained mathematicians such as Euler, Descartes, Cauchy and Lhuilier who gave its final form $\mathcal{X} (S)= 2-2g $,  for what is now known as a smooth closed connected orientable surface $S$ of genus $g$.  
\textcolor{black}{Remarkably, this Euler characteristic determines precisely  the  closed surface up to homeomorphism.}

One can say that an entire field, Algebraic Topology was innaugurated
 by Henri Poincar{\'e}, inspired by this formula. To state it as simply as possible, let $K$ be a simplicial complex. 
 Poincar{\'e} considered  vector spaces $C_i$ (over $Z_2$), of $i$-chains on $K$, \textcolor{black}{where the sum is defined as the union minus the intersection of the $i$-chains}. He wanted to measure the presence of special $i$-chains, called $i$-cycles that were not the border of an $i+1$-chain. This was accomplished by taking the quotient space of the space of $i$-cycles of  $K$, $Z_{i}(K)$ by all the $i$-cycles that are boundaries of $i+1$-chains, $B_{i}(K)$. Thus, this quotient space is called the $i$-th homology of $K$ and denoted by $\displaystyle {H_{i}(K)= {Z_{i}(K)\over B_{i}(K)}}$. The rank of $H_{i}(K)$ is the {\it $i$-th Betti number} of $K$, $\beta_{i}(K)$. The Euler characteristic is defined as  the alternating sum of the Betti numbers of $K$.

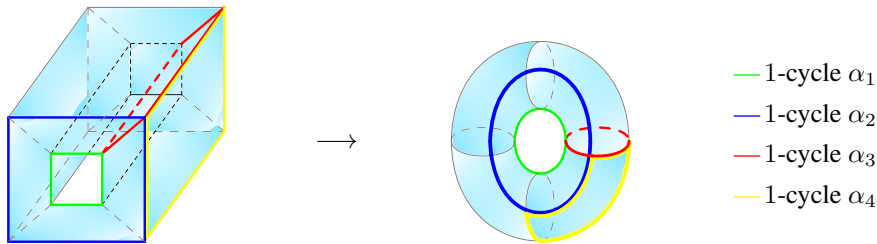
\begin{figure}[h]\label{fig:homology} 
\centering
  \begin{tikzcd}
  \begin{overpic}[align=c,scale=0.5]{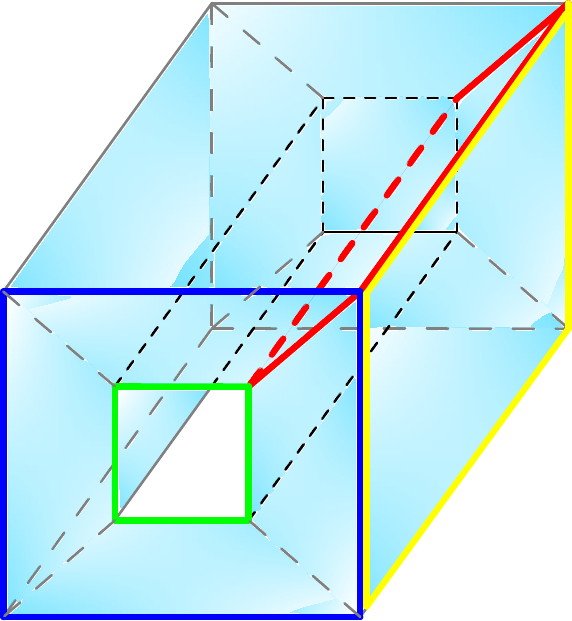}
  \put(130,40){$\longrightarrow$}
  \end{overpic} \hspace{3cm}
  \begin{overpic}[align=c,scale=0.5]{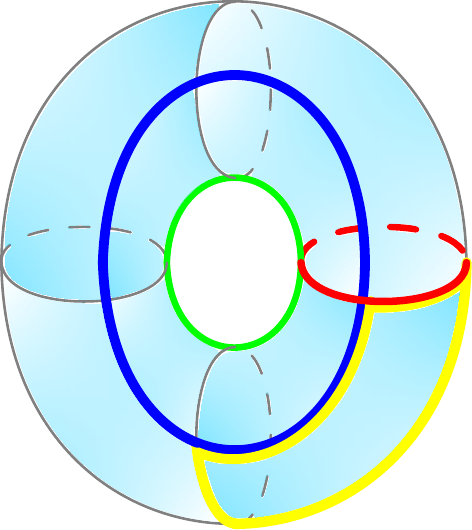}
  \put(140,80){\textcolor{green}{---}}
  \put(155,80){$1$-cycle $\alpha_{1}$}
  \put(140,60){\textcolor{blue}{---}}
  \put(155,60){$1$-cycle $\alpha_{2}$}
  \put(140,40){\textcolor{red}{---}}
  \put(155,40){$1$-cycle $\alpha_{3}$}
  \put(140,20){\textcolor{yellow}{---}}
  \put(155,20){$1$-cycle $\alpha_{4}$}
  \end{overpic} \hspace{5cm}
  \end{tikzcd}
\caption{Homology on a torus}
\end{figure}

In our $2$-dimensional context, $1$-cycles which belong to $Z_{1}(P)$  and are not in $B_{1}(P)$, play a very important role. See Figure \ref{fig:homology}.
Two $1$-cycles are equivalent, more precisely homologous, if they form the boundary of a $2$-chain.
 For instance, in Figure \ref{fig:homology}, $\alpha_{1}$, $\alpha_{2}$, $\alpha_{3}$ and $ \alpha_{4}$ all belong to $Z_{1}(P)$. However, $\alpha_{1}$ 
and $\alpha_{2}$ are homologous and  $ \alpha_{4}$ is also in $B_{1}(P)$. Hence, $\{\alpha_{1},\alpha_{3}\}$ is a basis of $\displaystyle {H_{1}(P)= {Z_{1}(P)\over B_{1}(P)}}$ 
 and consequently, $\beta_{1}(P)= \text{rank} H_{1}(P)=2$.
 Clearly, all $0$-cycles are homologous if $P$ is connected. Hence, $\beta_0(P) = \text{rank} H_0(P)=1$. Also, there is  only one $2$-cycle, $P$ itself, that forms the basis of $H_2(P)$.  Hence, 
$\beta_2 (P)= \text{rank} H_2(P)=1$. The Euler characteristic of $P$ is
$    \mathcal{X}(P) \ = \ \beta_{0}(P) - \beta_{1}(P) + \beta_{2}(P)=0$.

 It can be shown that for a smooth connected surface $S$ with $g$ handles,  $\beta_0(S)=\beta_2(S)=1$ and that
each handle contributes with two 1-cycles to the basis of $H_1(S)$.  Hence, $\beta_1(S)=2g$ and  $\mathcal{X}(S) \ = \ \beta_{0}(S) - \beta_{1}(S) + \beta_{2}(S)=2-2g.$

{
Two important properties of the Euler characteristic that will be used henceforth are:\\
\begin{itemize}
\item[i) ][Homotopy invariance] Let $A$ and $B$ be two homotopically equivalent spaces. Then, one has: $$\mathcal{X}(A) \ \ = \ \ \mathcal{X}(B).$$
\item[ii) ][Inclusion-exclusion principle] Let $A$ and $B$ be any two closed sets. Then, the following equality holds: $$\mathcal{X}(A \cup B) \ \ = \ \ \mathcal{X}(A) + \mathcal{X}(B) - \mathcal{X}(A \cap B).$$
\end{itemize}
}

\textcolor{black}{In what follows, the overarching idea is to understand the topology of a singular manifold by studying a family of smooth manifolds that degenerate to it.}

 A very \textcolor{black}{well-known} and elementary example of passing from a smooth surface to a singular surface, is the family of surfaces obtained from the inverse images of the function  $f: \mathbb{R}^{3} \rightarrow \mathbb{R}$, given by $f(x,y,z) = x^{2} + y^{2} - z^{2}$. Note that $f^{-1}(1)$ is a smooth surface, \textcolor{black}{a one-sheet} hyperboloid, while $f^{-1}(0)$ is a singular surface, more specifically a double cone. By considering the surfaces $f^{-1}(t)$ obtained by varying $t$ continuously from $t = 1$ to $t = 0$, one can see a circle whose radius is decreasing until the circle degenerates into a point at the level curve  $z = 0$. This contraction is responsible for the birth of the singular cone point and consequently of the singular surface. 
 
 One can visualize a similar situation in a polyhedron setting. See Figure \ref{fig:polin}.
 \vspace{0.2cm}

\begin{figure}[h] 
  \begin{tikzcd} \hspace{0.2cm}
  \begin{overpic}[align=c,scale=0.6]{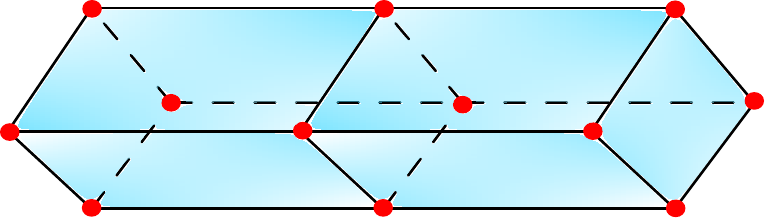}
  \put(10,35){$P$}
  \end{overpic} \hspace{0.8cm}
    \begin{overpic}[align=c,scale=0.6]{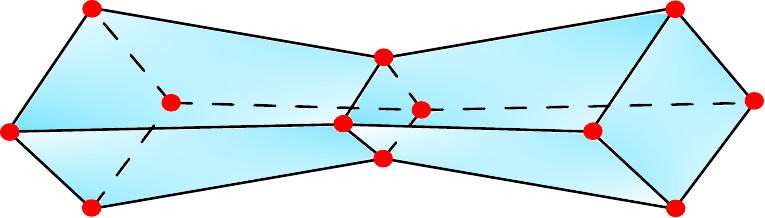}
    \end{overpic} \hspace{0.8cm}
    \begin{overpic}[align=c,scale=0.6]{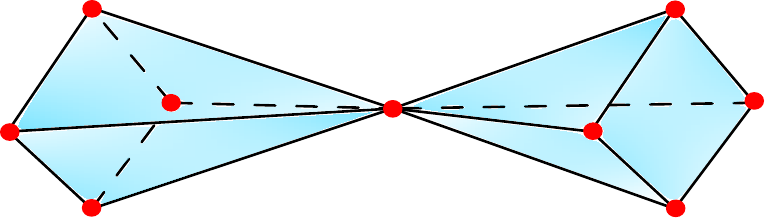}
    \put(86,35){$P'$}
    \end{overpic}
  \end{tikzcd}
\caption{Polyhedron $P$  (left-most) collapsing to polyhedron $P'$ (right-most).}\label{fig:polin}
\end{figure}

It is quite interesting to see the effect that this degeneracy has on the  Euler characteristic . 
The collapsing of the middle one cycle in $P$ to a vertex has the net effect of removing three vertices and four edges from the formula \textcolor{black}{$\mathcal{X} (P)= V - E + F $, where $V$, $E$, and $F$ are, respectively, the number of vertices, edges, and  faces of the polyhedron $P$. See Figure~\ref{fig:polin}. }
$$\mathcal{X}(P') = V' - E' + F' = (V - 3) - (E - 4) + F = V - E + F + 1 = \mathcal{X}(P) + 1.$$

In this work, we will consider this contraction and refer to it as a \textcolor{black}{\it  collapsing operation}. Two other operations on the images of loops on smooth surfaces are considered\textcolor{black}{: {\it zipping} and {\it double loop identification}}  \textcolor{black}{both of} which produce singular surfaces.

Let $M$ be a compact connected orientable surface.  The surface $M$ is of type $(g,b)$ if it has genus $g$ and $b$ boundary components and is denoted by $M_{g,b}$. If $b = 0$, the surface is denoted by $M_{g}$ \textcolor{black}{ and called a {\it closed} surface}. A \textit{loop} in $M$ is a \textcolor{black}{smooth} map $\alpha: S^{1} \rightarrow M$, which is identified to its image in  $M$. All loops will be orientation preserving. A loop is \textit{simple} if 
$\alpha$ is injective, that is,  $\alpha$ has no self-intersection. A loop is \textit{trivial} in $M$ if it is homotopic to a point.  Two loops are {\it cobordant} if the two loops bound a subsurface.

Suppose that the smooth closed surface $M$ is embedded  in $\mathbb{R}^{3}$,  \textcolor{black}{thus orientable}.  This embedding of $M$ partitions $\mathbb{R}^{3}$ into a bounded region $\mathbb{I}$ and an unbounded region $\mathbb{O}$ so that 
$\mathbb{I} \cap \mathbb{O} = M$ and such that $\mathbb{I} \cup \mathbb{O} = \mathbb{R}^{3}$. 

\newtheorem{Def}{Definition}

\begin{Def}
A simple loop $\alpha: S^{1} \rightarrow M$ is called a \textbf{handle loop} if it is trivial in the homology of $\mathbb{I}$ and non-trivial in the homology of $\mathbb{O}$. A \textbf{tunnel loop} is trivial in the homology of $\mathbb{O}$ and  non-trivial in the homology of $\mathbb{I}$.  
\textcolor{black}{ Whenever $(M - \alpha)$ is not connected, $\alpha$ is called a  \textbf{separating loop}. In this case, we say that $\alpha$ splits $M_{g}$ in two disjoint subsurfaces of genus $k$ and $g - k$.}
 
\end{Def}

\begin{figure}[h] 
\centering
  \begin{tikzcd} \hspace{0.2cm}
  \begin{overpic}[align=c,scale=0.4]{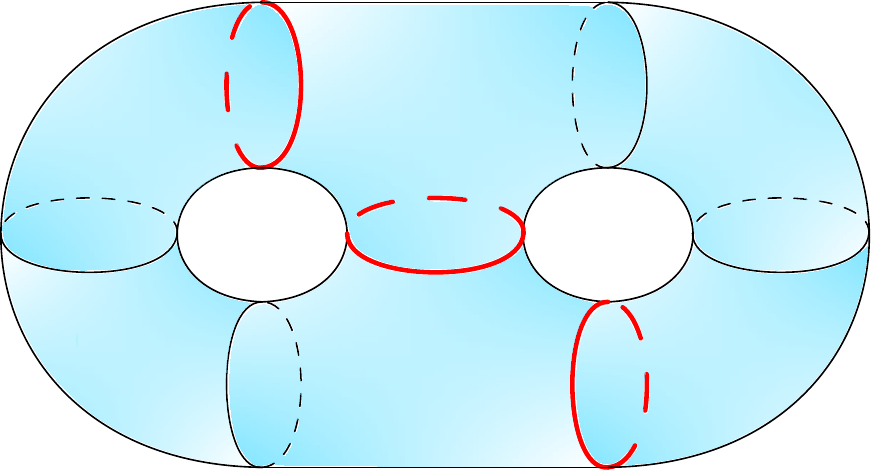}
  \end{overpic} \hspace{1cm}
    \begin{overpic}[align=c,scale=0.4]{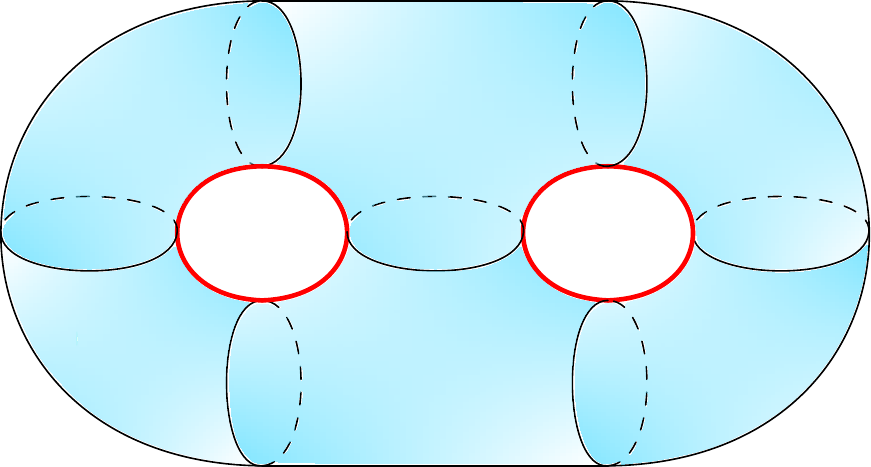}
    \end{overpic} \hspace{1cm}
    \begin{overpic}[align=c,scale=0.4]{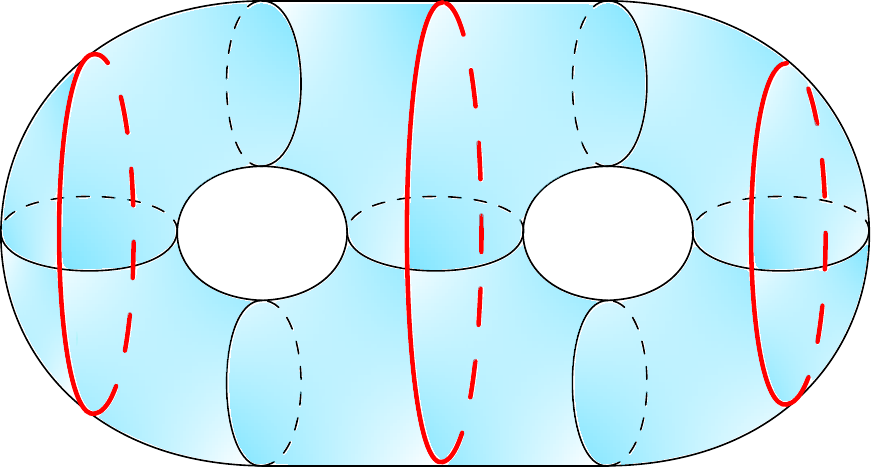}
    \end{overpic}
  \end{tikzcd}
\caption{Examples of handle (left-most), tunnel (center) and separating loops (right-most) in red.}
\end{figure}

\subsection{Simple loop operations}

 \textcolor{black}{Now} we define  \textcolor{black}{the} operations that can be performed on simple loops to create singular  \textcolor{black}{surfaces}.  

\begin{Def}\label{def:operacoes}
Let $\alpha: S^{1} \rightarrow M$ and $\beta: S^{1} \rightarrow M'$ be simple loops  \textcolor{black}{each of which are} either separating, handle or tunnel; and $M$ and $M'$ are smooth closed orientable surfaces, possibly the same. Define  \textcolor{black}{the following operations, which will be called \textbf{simple loop operations}}: 

\begin{enumerate}
    \item \textbf{Collapsing} of $\alpha$: consider a  disk $D$, up to homeomorphism, bounded by $\alpha$ such that it  
    is contractible to a point $p$ in the complement of $M$ in $\mathbb{R}^3$. The collapsing of  $\alpha$ is the retraction of $D$ to $p$.

    \begin{figure}[htb]
        \centering
        \includegraphics[scale=0.45]{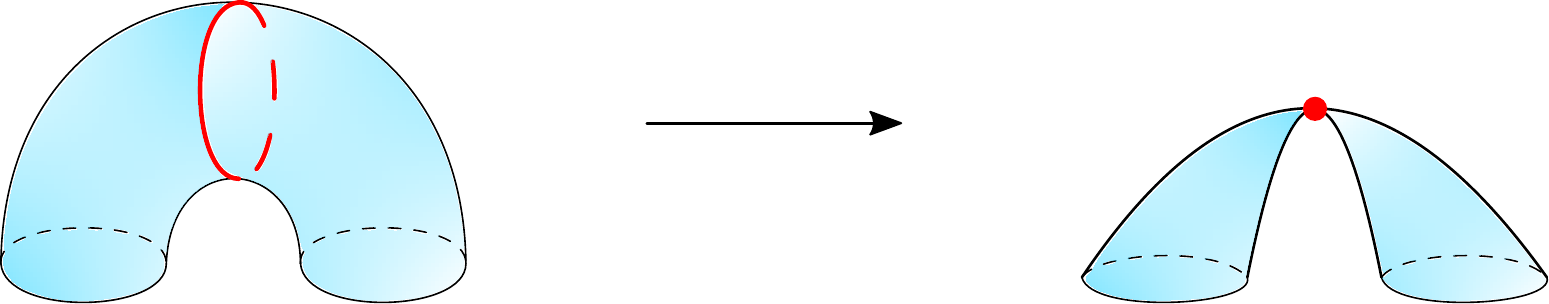}
        \caption{Example of collapsing}
        \label{fig:collapsing}
    \end{figure}
    
    \item \textbf{Zipping} of $\alpha$: consider a  disk $D$, up to homeomorphism, bounded by $\alpha$ such that it is contractible to a curve $d$ joining two distinct points on $\alpha$ in the complement of $M$ in $\mathbb{R}^3$. The zipping of $\alpha$ is the retraction of $D$ to $d$.

    \begin{figure}[htb]
        \centering
        \includegraphics[scale=0.45]{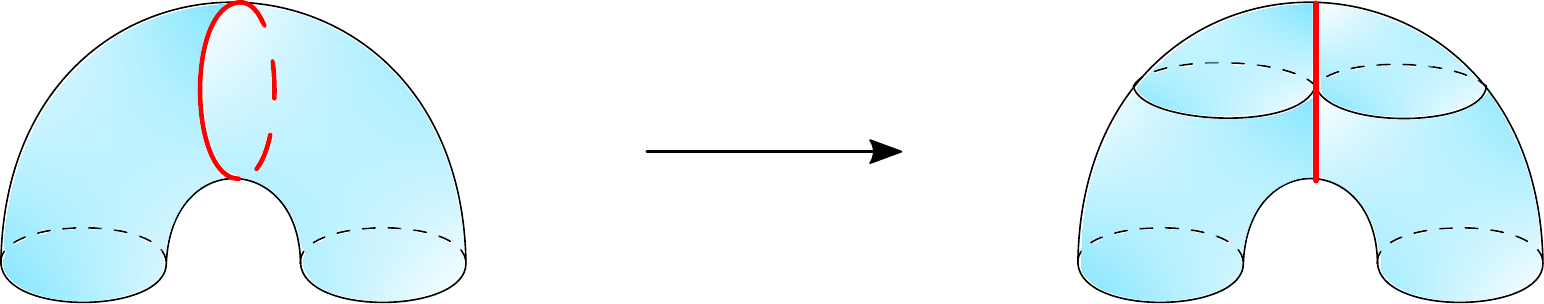}
        \caption{Example of zipping}
        \label{fig:folding}
    \end{figure}
    
    \item \textbf{Double loop identification} of $\alpha$ and $\beta$: the loops $\alpha$ and $\beta$ are identified, via some orientation preserving  homeomorphism $h: \alpha \rightarrow \beta$;

    \begin{figure}[htb]
        \centering
        \includegraphics[scale=0.45]{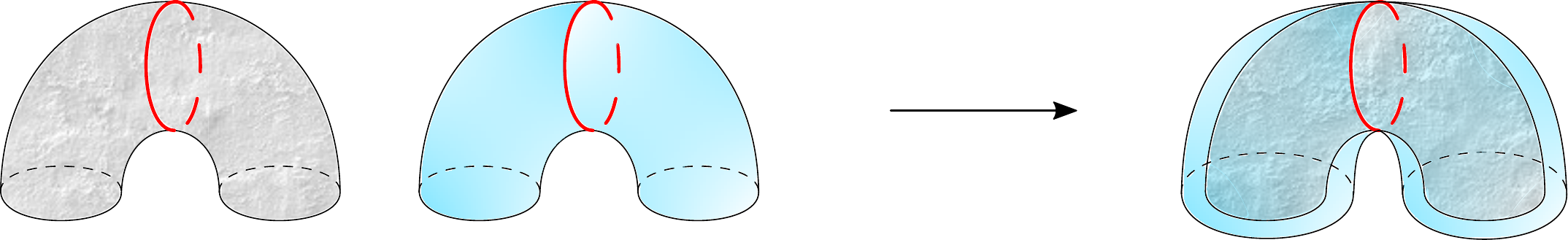}
        \caption{Example of double loop identification}
        \label{fig:identification}
    \end{figure}
    
\end{enumerate}
\end{Def}

Once a specific loop operation is performed, the resulting singularity or  singular set can be easily identified and vice-versa.
 \textcolor{black}{Typically}, the operations that are chosen in a singularization process, see Definition \ref{def-4}, are based on the  \textcolor{black}{type of non-manifold set components}
  that one wants the singular surface to possess. 

In this manner, note that by applying the  \textcolor{black}{collapsing operation} one obtains surfaces of  revolution such as an \textit{eight surface} (figure \ref{fig:eight}) and a \textit{horn torus} (figure \ref{fig:horn}), the latter appears as a cyclical model of the Universe.  

\begin{figure}[htb]
	\centering
	\begin{minipage}[b]{0.4\linewidth}
		 \begin{equation*}
            \begin{tikzcd}
                \includegraphics[scale=0.25]{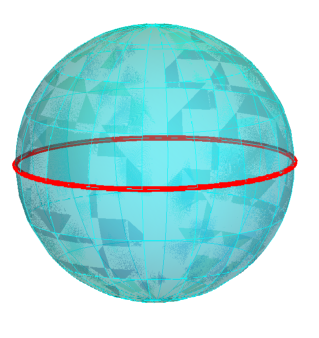} \arrow{rr}{Collapsing}  & & \includegraphics[scale=0.25]{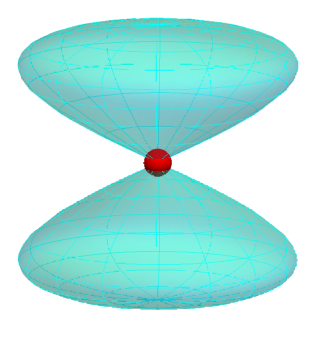} 
            \end{tikzcd}
        \end{equation*}
		\caption{Eight Surface}	
		\label{fig:eight}
	\end{minipage}
	\hfill
	\begin{minipage}[b]{0.5\linewidth}
		\begin{equation*}
            \begin{tikzcd}
                \includegraphics[scale=0.27]{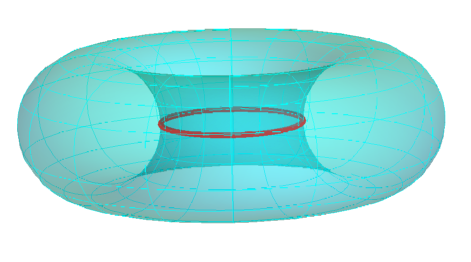} \arrow{rr}{Collapsing}  & & \includegraphics[scale=0.27]{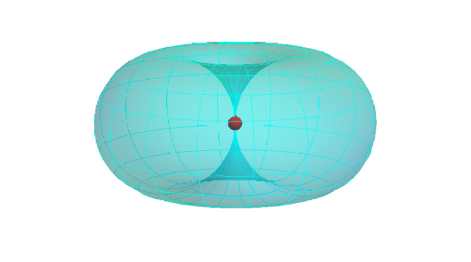} 
            \end{tikzcd}
        \end{equation*}
		\caption{Horn Torus}
		\label{fig:horn}
	\end{minipage}
\end{figure}

On the other hand, the \textcolor{black}{zipping operation} appears in the family of surfaces parametrized by $\phi(u,v) = ((a+b*cos(v))*cos(u), (a + b*cos(v))*sin(u),b*sin(v)*cos(ku))$, for $a > b > 0$. The number of \textit{folds} present in a surface of this family varies for different values of $k$:

\begin{figure}[h!]
	\begin{minipage}[b]{0.3\linewidth}
		 \begin{equation*}
            \begin{tikzcd}
                \includegraphics[scale=0.26]{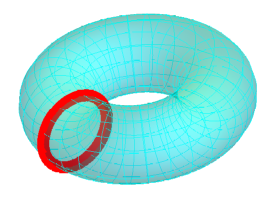} \arrow{r}{Zipping}  & \includegraphics[scale=0.26]{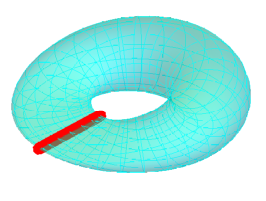} 
            \end{tikzcd}
        \end{equation*}
		 \caption{$\phi(u,v)$ with $k = 0.5$}	
		\label{fig:first_sine_torus}
	\end{minipage}
	\hfill
	\begin{minipage}[b]{0.59\linewidth}
		\includegraphics[scale=0.28]{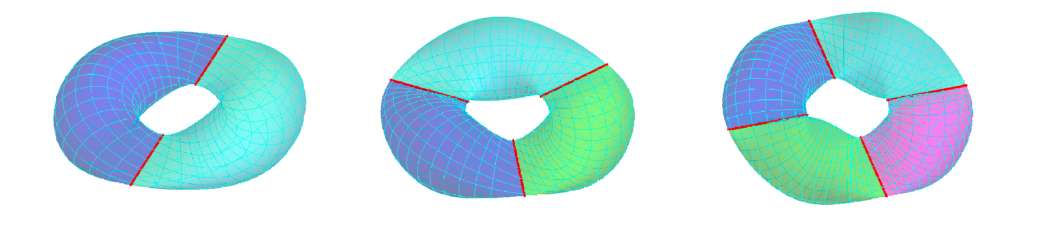}
		\caption{$\phi(u,v)$ with $k = 1, \ k = 1.5, \ k = 2$ (resp.)}
		\label{fig:final_sine_torus}
	\end{minipage}
\end{figure}

The tori chain in Figure~\ref{fig:torichain} illustrates a singular surface obtained by double loop identifications.

\begin{figure}[htb!]
    \begin{equation*}
            \begin{tikzcd}
                \includegraphics[scale=0.17]{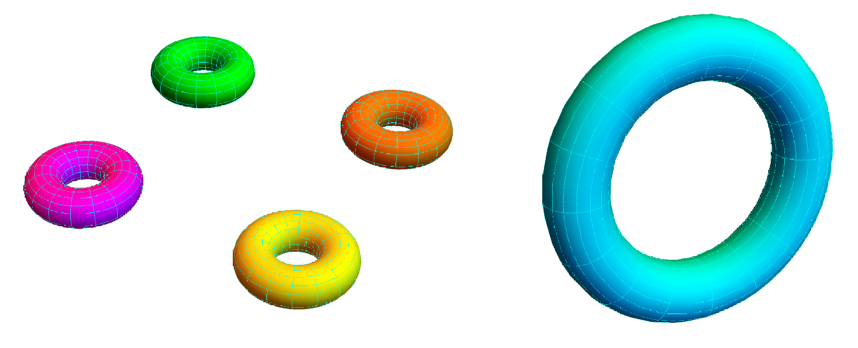} \xrightarrow[identification]{Double \ loop} \includegraphics[scale=0.17]{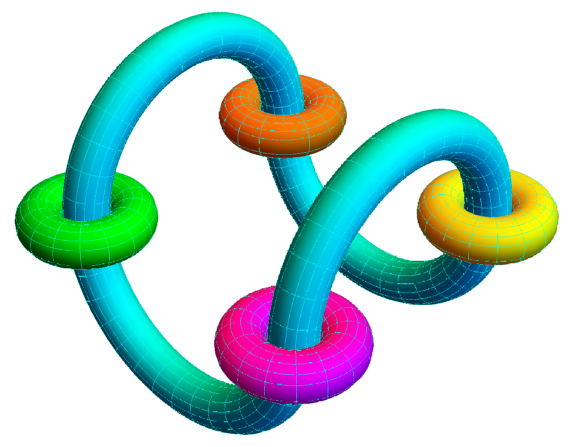} 
            \end{tikzcd}
        \end{equation*}
		\caption{Tori chain}
		\label{fig:torichain}
\end{figure}

\begin{Def}\label{def:loop_col}
Given a smooth, closed, connected, orientable surface $M_g$, let $\mathcal{L}(M_{g})$ be a collection of disjoint handle, tunnel, or separating loops in $M_g$ with a  simple loop operation assigned to each one. This  definition  can be extended   to a  disjoint union of surfaces, $\sqcup_{i=1}^{n}{M_{g_{i}}}$. 
\end{Def}

For a  finite collection of smooth connected orientable closed surfaces, $\{M_{g_{i}} \ | \ i = 1, \ldots, n\}$, define $G = \sum_{i=1}^{n}{g_{i}}$ as the {\it genus} of  $\ \sqcup_{i=1}^n{M_{g_{i}}}$.   

\begin{Def}\label{def-4}
A   \textbf{singularization} of a finite collection $\{M_{g_{i}} \ | \ i = 1, \ldots , n\}$ of smooth, closed, connected, orientable surfaces will be attained from   $\mathcal{L}(\sqcup_{i=1}^{n}{M_{g_{i}}}) $   by performing the simple loop operations assigned therein, and denoted by  $S(\sqcup_{i=1}^{n}{M_{g_{i}}})$.
\end{Def}

Examples of \textcolor{black}{such} singularization can be seen in Section~\ref{sec:Euler_formulas}.

We also note that the resulting singular surface $S(\sqcup_{i=1}^{n}{M_{g_{i}}})$ may not be connected.

\section{\textcolor{black}{Effect of Loop Operations  on the Euler  Characteristic}}\label{sec:Euler_formulas}

In this section we will study the effect that the singularization of a smooth surface $M_g$ (respectively, a disjoint union of smooth surfaces $\sqcup_{i=1}^{n}{M_{g_{i}}}$) has on the Euler characteristic of  $M_g$ (respectively, $\sqcup_{i=1}^{n}{M_{g_{i}}}$).

The next theorem \textcolor{black}{describes the effect on the original smooth surface's Euler characteristic after the simple loop operations are performed. In other words, Theorem \ref{teo:forte}  describes how $\mathcal{X}(S)$ can be computed from $\mathcal{X}(M_g)$.
A collapsing or zipping operation  adds one to the  Euler characteristic $\mathcal{X}(M_g)$, whereas  a double loop identification leaves it unchanged.}

\newtheorem{Teo}{Theorem}    
   
\begin{Teo}\label{teo:forte}
Let $S(\sqcup_{i=1}^{n}{M_{g_{i}}})$ be a singularized surface obtained from the collection $\mathcal{L}(\sqcup_{i=1}^{n}{M_{g_{i}}}) $. Then the Euler characteristic of $S(\sqcup_{i=1}^{n}{M_{g_{i}}})$ is independent of the number of double loop identifications performed and is equal to: $$\mathcal{X}(S(\sqcup_{i=1}^{n}{M_{g_{i}}})) \ = \ \left(\sum_{i=1}^{n}{\mathcal{X}(M_{g_{i}})}\right) + C + Z \ = \ 2n - 2G + C + Z,$$ where $C$ is the number of  collapsing  operations, $Z$ is the number of  zipping  operations and $G = \displaystyle\sum_{i=1}^{n}{g_{i}}$.
\end{Teo}

The proof of Theorem \ref{teo:forte} will follow from a series of lemmas, each of which will prove the effect on the Euler characteristic of $\sqcup_{i=1}^{n}{M_{g_{i}}}$ after performing a specific type of  operations, i.e. collapsing, zipping or double loop identification, on a collection of loops $\mathcal{L}(\sqcup_{i=1}^{n}{M_{g_{i}}})$.

\newtheorem{Lem}{Lemma}

\begin{Lem}[\textcolor{black}{Collapsing and  Zipping}]\label{lem:EulerCone}
Let $S(M_{g})$ be a singularized surface originating from $M_{g}$ by performing   \textcolor{black}{$C$ collapsings and $Z$ zippings. Then } $$\mathcal{X}(S(M_{g})) = 2 - 2g + C + Z.$$
\end{Lem}

\begin{proof}

The idea behind the proof is to show that by performing a collapsing operation, as well as, a zipping operation on a loop in $\mathcal{L}(M_{g})$ the effect on the Euler characteristic, $\mathcal{X}(M_{g})$, will be an increase by one.

\begin{itemize}
     \item[a) ] First, note that collapsing a loop $\alpha \in \mathcal{L}(M_{g})$, where $\alpha$ is:  
\begin{itemize}
    \item[i) ] a \textbf{separating simple loop}, increases $\beta_{2}(M_{g})$ by one;
    \item[ii) ] a \textbf{tunnel loop}, decreases $\beta_{1}(M_{g})$ by one;
    \item[iii) ] a \textbf{handle loop}, decreases $\beta_{1}(M_{g})$ by one, unless $\alpha$ is cobordant to another handle loop $\beta \in \mathcal{L}(M_{g})$ being collapsed. Recall that if $\alpha$ and $\beta$ are cobordant, then there is a subsurface $N_{k,2} \subset M_{g}$, for some $k \in \{1, \ldots, g \}$, such that $\partial(N_{k,2}) = \alpha \cup \beta$. Thus, collapsing $\alpha$ and $\beta$ transforms $N_{k,2}$ in a closed subsurface, so that one of them decreases $\beta_{1}(M_{g})$ by one, and the other increases $\beta_{2}(M_{g})$ by one. Something similar occurs if $n$ cobordant handle loops in $\mathcal{L}(M_{g})$ are collapsed.

\end{itemize}

Since the Euler characteristic of $M_{g}$ is given by the alternating sum: $$\mathcal{X}(M_{g}) = \beta_{0}(M_{g}) - \beta_{1}(M_{g}) + \beta_{2}(M_{g}),$$ the net effect of decreasing  $\beta_{1}(M_{g})$ by one as well as increasing $\beta_{2}(M_{g})$ by one is the increase of  $\mathcal{X}(M_{g})$ by one. So, if $S(M_{g})$ is obtained by $C$ collapsing operations, it follows that: $$\mathcal{X}(S(M_{g})) \ \ = \ \ \mathcal{X}(M_{g}) + C \ \ = \ \ 2 - 2g + C;$$

\item[b) ] Now, note that zipping a loop $\alpha \in \mathcal{L}(M_{g})$ is homotopically equivalent to collapsing $\alpha$. The homotopy simply contracts a segment to a point. In other words, the following diagram is commutative:

\begin{equation}\label{diagrama}
  \begin{tikzcd}
    M_{g} \arrow{rr}{Z \ zippings} \arrow[swap]{rrdd}{Z \ collapsings} & & S(M_{g}) \arrow{dd}{Homotopy} \\
    & & \\
     & & \widetilde{S}(M_{g})
  \end{tikzcd}
\end{equation}

Consequently, by item a) and Euler characteristic's invariance under homotopy, if $S(M_{g})$ originates from $M_{g}$ by performing $Z$ zippings, we have that $\mathcal{X}(S(M_{g})) = 2 - 2g + Z$.

\item[c) ] Finally, once the loops in $\mathcal{L}(M_{g})$ are all disjoint, if a total of $C$ collapsings and $Z$ zippings are performed in the singularization of $M_{g}$, it follows from a) and b) that $\mathcal{X}(S(M_{g})) = 2 - 2g + C + Z$.

\end{itemize}

\end{proof}

\newtheorem{Cor}{Corolário}

\begin{figure}[h]
\begin{equation*}
  \begin{tikzcd}
   \includegraphics[scale=0.4]{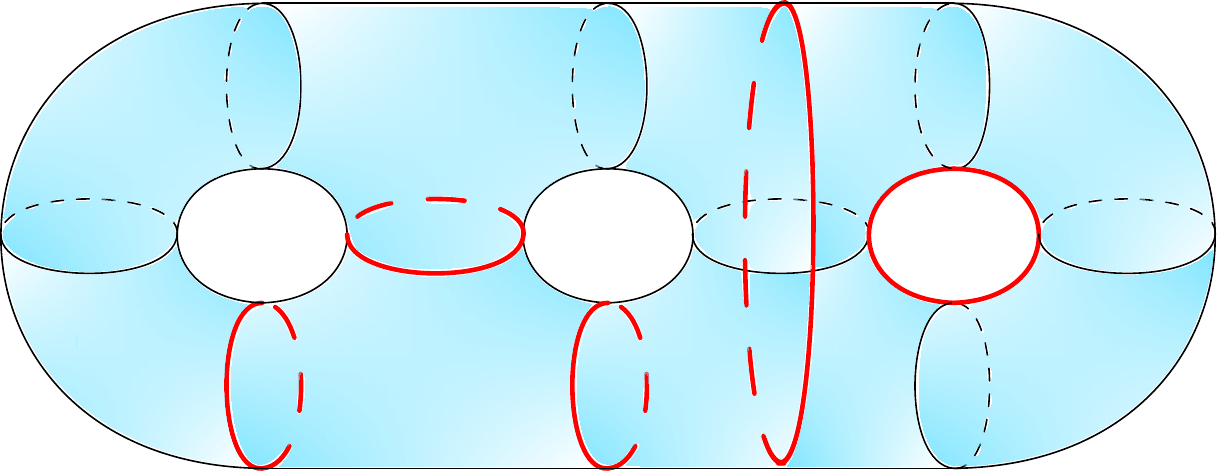} \arrow{rr}{Zipping} \arrow[swap]{rrdd}{Collapsing} & & \includegraphics[scale=0.4]{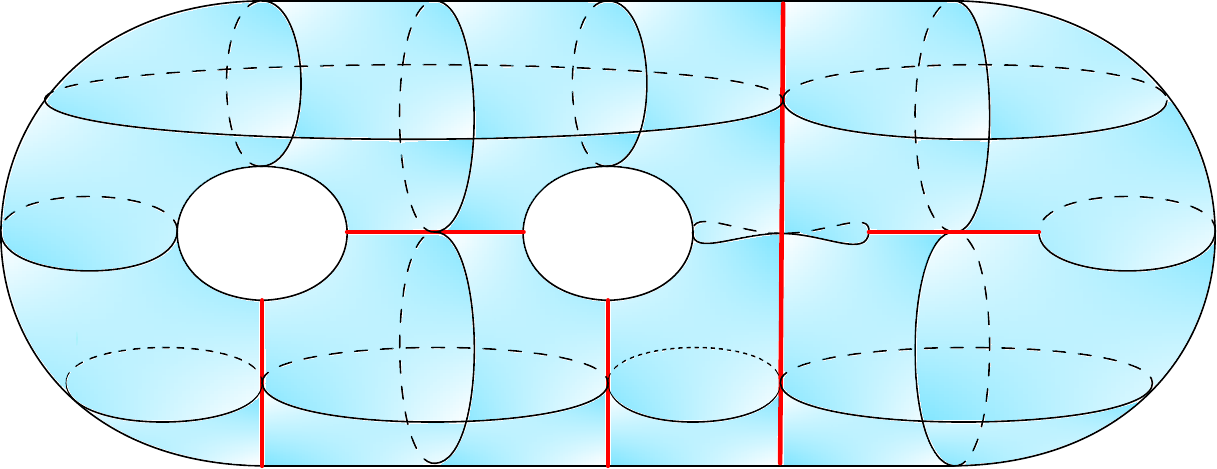} \arrow{dd}{Homotopy} \\
    & & \\
     & & \includegraphics[scale=0.4]{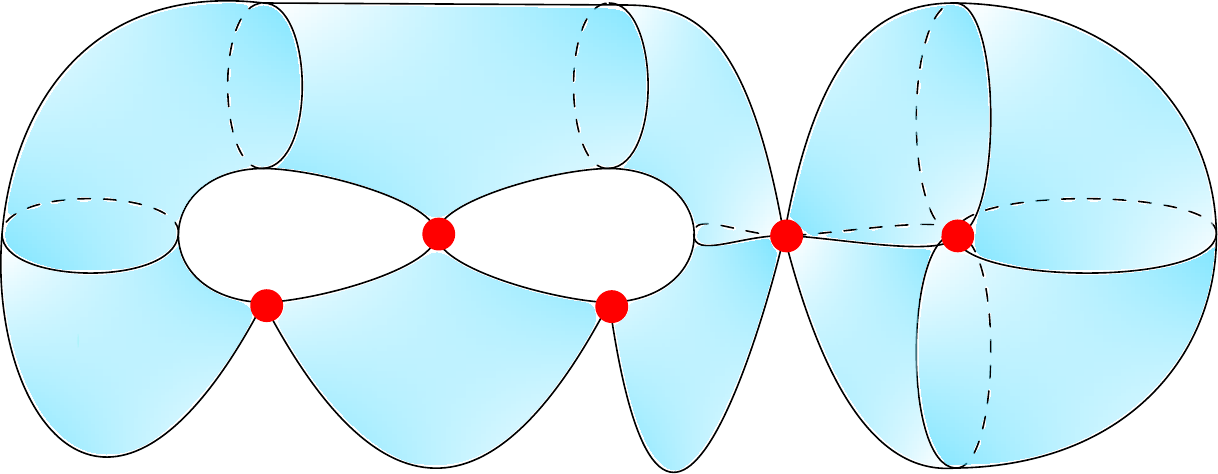}
  \end{tikzcd}
 \end{equation*} 
\caption{Singularization via Collapsing and Zipping Operations}
\label{diagrama2}
\end{figure}

In Figure \ref{diagrama2}, a singularization via zipping and collapsing operations of a genus $g=3$ surface, a $3$-torus, is presented. Note that the surfaces are homotopy equivalent. Hence the Euler characteristic for both singularized surfaces are the same, that is,  $\mathcal{X}(\widetilde{S}(M_{3})) = 2 - 2\times3 + 5=1$.

It is easy to see that Lemma~\ref{lem:EulerCone} generalizes whenever the singular surface is obtained from a collection $\{M_{g_{i}} \ | \ i = 1, \ldots, n\}$ of smooth surfaces by performing $C$  collapsing operations and $Z$  zipping operations on $\mathcal{L}(\sqcup_{i = 1}^{n}{M_{g_{i}}})$. In this case, $$\mathcal{X}(S(\sqcup_{i = 1}^{n}{M_{g_{i}}})) \ = \ 2n - 2G + C + Z,$$ where $G = \displaystyle\sum_{i=1}^{n}{g_{i}}$.

The next couple  of lemmas, Lemmas \ref{lem:duplo} and \ref{lem:duplo2},  will prove that double loop identification on $\mathcal{L}(\sqcup_{i = 1}^{n}{M_{g_{i}}})$ does not  change  the Euler characteristic given by $\sum_{i = 1}^{n}{\mathcal{X}(M_{g_{i}})}$.

\begin{Lem}[Double loop identification]\label{lem:duplo}
Let  $S(\sqcup_{i = 1}^{n}{M_{g_{i}}})$ be a singular surface obtained from double loop identifications on $\mathcal{L}(\sqcup_{i = 1}^{n}{M_{g_{i}}})$ where each double loop identification is performed on a pair of loops lying  on different surface components.  Then the Euler characteristic of $S(\sqcup_{i = 1}^{n}{M_{g_{i}}})$ remains the same, that is,  $$\mathcal{X}(S(\sqcup_{i = 1}^{n}{M_{g_{i}}})) \ = \ \sum_{i = 1}^{n}{\mathcal{X}(M_{g_{i}})}  
\ = \ 2n - 2G.$$
\end{Lem}

\begin{proof}

Since the singular surface $S(\sqcup_{i = 1}^{n}{M_{g_{i}}})$ is a union of smooth surfaces $\sqcup_{i = 1}^{n}{M_{g_{i}}}$, intersecting along loops in $\mathcal{L}(\sqcup_{i = 1}^{n}{M_{g_{i}}})$, it follows from the inclusion-exclusion principle that: 
\begin{equation}
   \mathcal{X}(S(\sqcup_{i = 1}^{n}{M_{g_{i}}})) \ = \ \sum_{i = 1}^{n}{\mathcal{X}(M_{g_{i}})} - \sum_{i = 1}^{D}{\mathcal{X}(S^{1})}, 
\end{equation}

where $D$ is the number of double loop identifications. Since $\mathcal{X}(S^{1}) = 0$, the Euler characteristic of $S(\sqcup_{i = 1}^{n}{M_{g_{i}}})$ does not depend on the number of double loop identifications performed and is given by:  $$\mathcal{X}(S(\sqcup_{i = 1}^{n}{M_{g_{i}}})) \ = \ \sum_{i = 1}^{n}{\mathcal{X}(M_{g_{i}})}  
\ = \sum_{i = 1}^{n}{2 - 2g_{i}} \ = \ 2n - 2G.$$
\end{proof}

\paragraph{Example} In Figure~\ref{fig:tangled}, an example of singularization via double loop identifications on disjoint surfaces is presented. The collection $\{ M_{2}^1, M_{2}^2, M_{1}\}$ contains two bi-tori $M_{2}^1$ and $M_{2}^2$ and a torus $M_{1}$, as well as, a family of loops identifications $\mathcal{L}$. The  singular surface $S(\sqcup_{i = 1}^{3}{M_{g_{i}}})$, according to Lemma~ \ref{lem:duplo}, has Euler characteristic equal to: $$\mathcal{X}(S(\sqcup_{i = 1}^{3}{M_{g_{i}}})) \ = \ 2n - 2G \ = \ 2\times3 - 2\times5 \ = \ -4.$$

\begin{figure}[ht]
    \begin{equation*}
            \begin{tikzcd}
                \includegraphics[scale=0.4]{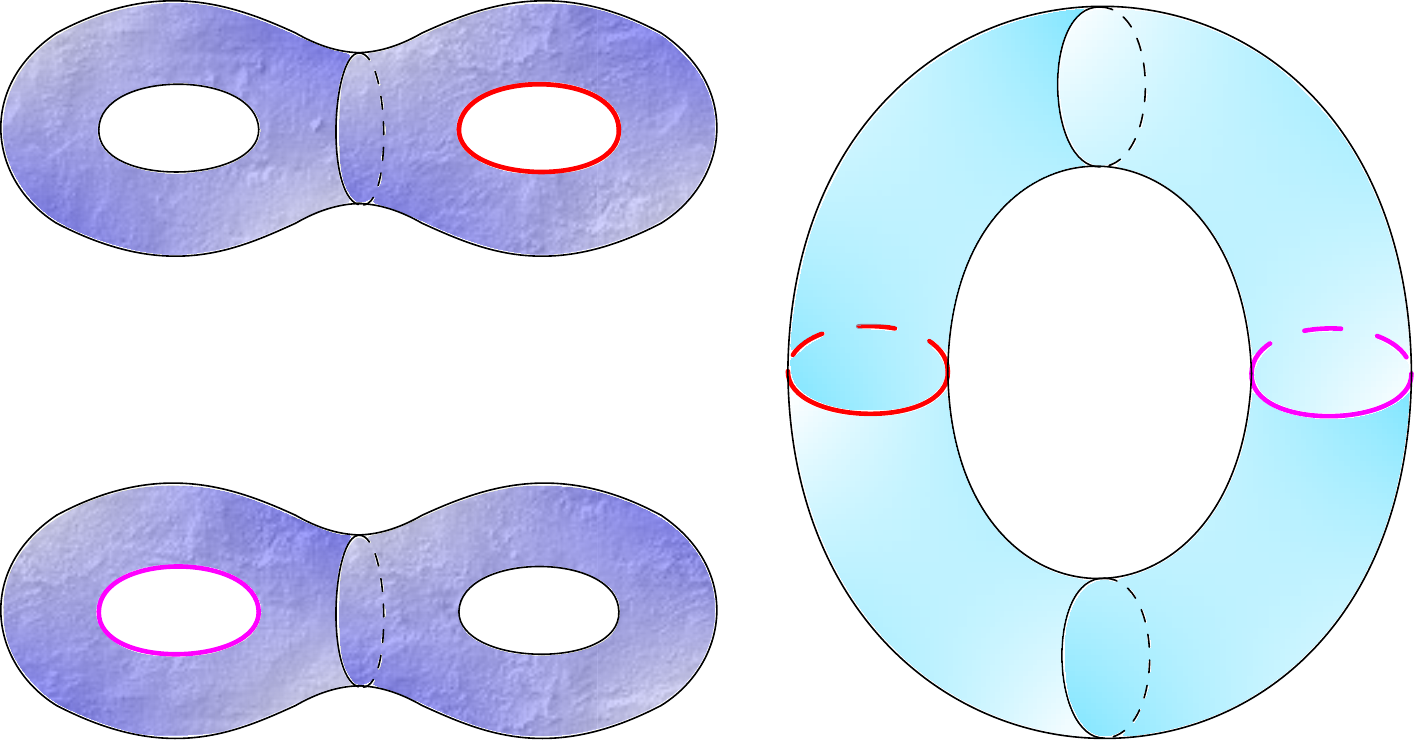} \xrightarrow[identification]{Double \ loop}  \includegraphics[scale=0.4]{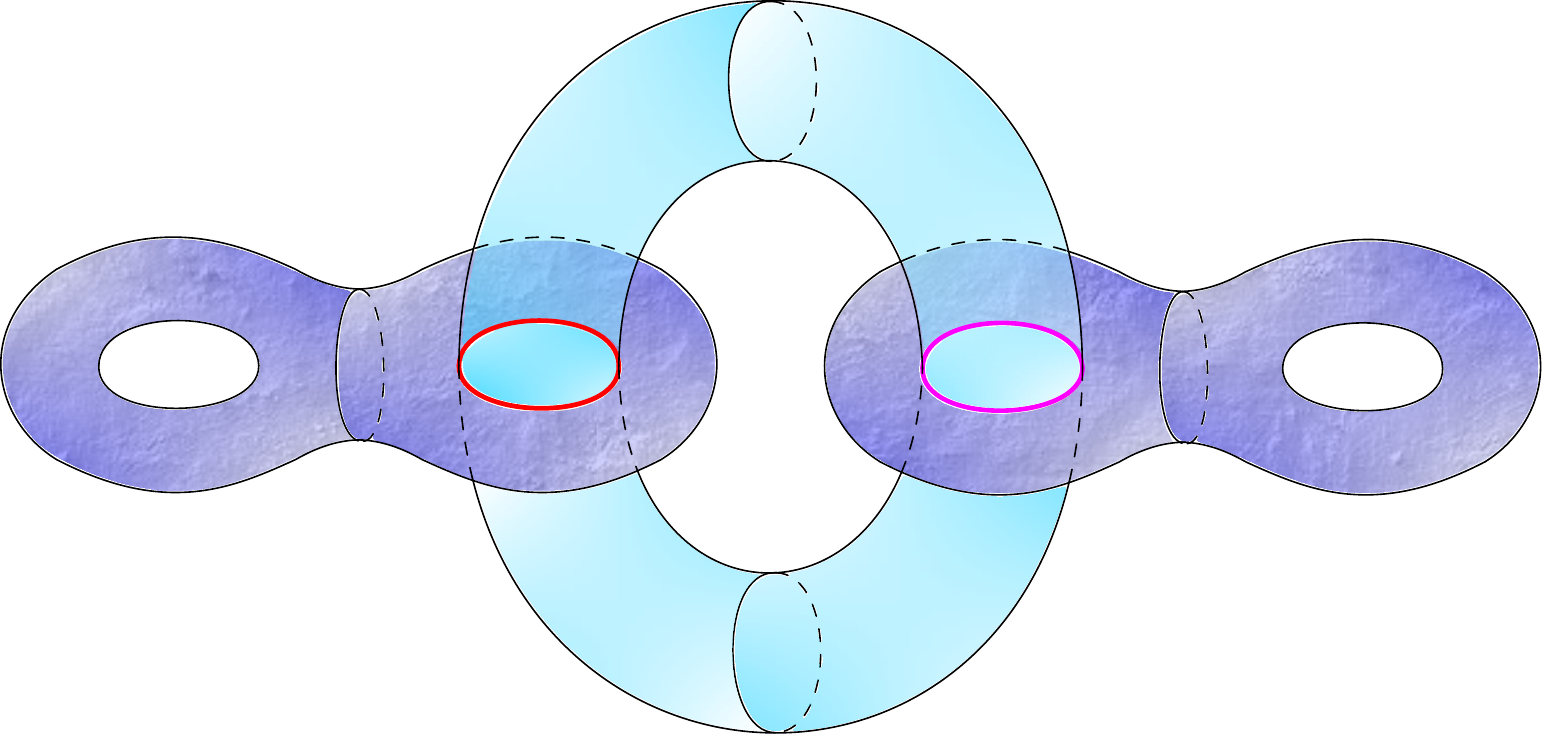} 
            \end{tikzcd}
        \end{equation*}
		\caption{Surface chain}
		\label{fig:tangled}
\end{figure}

The following lemma will prove that the invariance of the Euler characteristic under  double loop identifications still holds when the loops are chosen on the same smooth surface.

\begin{Lem}[Double loop identification]\label{lem:duplo2}
Let $\alpha$, $\beta: S^{1} \rightarrow M_{g}$ be two disjoint simple loops and $S(M_{g})$  the singular surface obtained by a double loop identification of $\alpha$ and $\beta$. Then $$\mathcal{X}(S(M_{g})) = \mathcal{X}(M_{g})$$
\end{Lem}

\begin{proof}

\textcolor{black}{Either the two loops  $\alpha$ and $\beta$ are cobordant, or they are not. We consider both cases.}
 Technically, this distinction is not needed, however the cobordant case presents a more topological description of the singular surface produced by the operation. 
\begin{itemize}
    \item[i) ] \textbf{$\pmb{\alpha}$ and $\pmb{\beta}$ are non-cobordant}:

\textcolor{black}{The quotient space $M_{g}/\alpha \sim \beta$ given by a double loop identification of $\alpha$ and $\beta$ is homotopically equivalent to gluing a cylinder $C_{\alpha}^{\beta}$ on $M_{g}$, where $\alpha$ and $\beta$ are  each glued to an end circle of $C_{\alpha}^{\beta}$. The homotopy is the contraction of the cylinder to a circle.\\
Hence it follows from the homotopy invariance of the Euler characteristic and the inclusion-exclusion principle that: $$\mathcal{X}(S(M_{g})) \ \ = \ \ \mathcal{X}(M_{g}) + \mathcal{X}(C_{\alpha}^{\beta}) - \mathcal{X}(S^{1} \cup S^{1}) \ \ = \ \ \mathcal{X}(M_{g})$$}

\begin{figure}[!h] 
\centering

  \begin{tikzcd} \hspace{0.2cm}
  \begin{overpic}[align=c,scale=0.4]{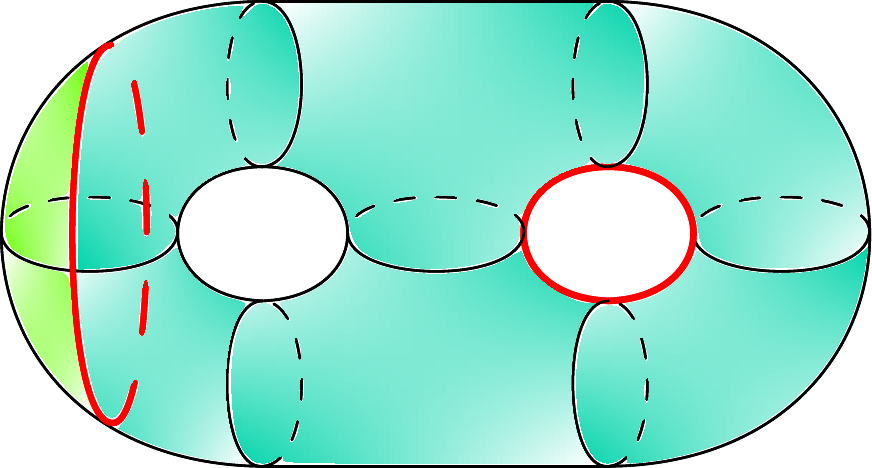}
  \end{overpic} \hspace{1cm}
    \begin{overpic}[align=c,scale=0.4]{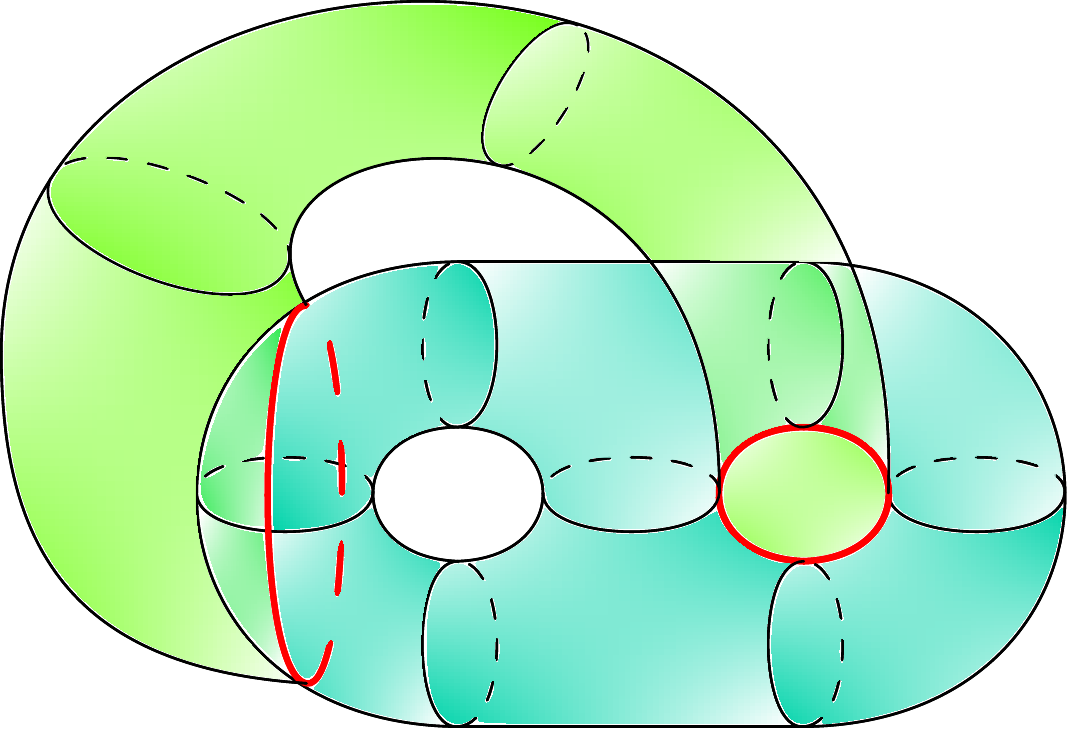}
    \end{overpic} \hspace{1cm}
    \begin{overpic}[align=c,scale=0.4]{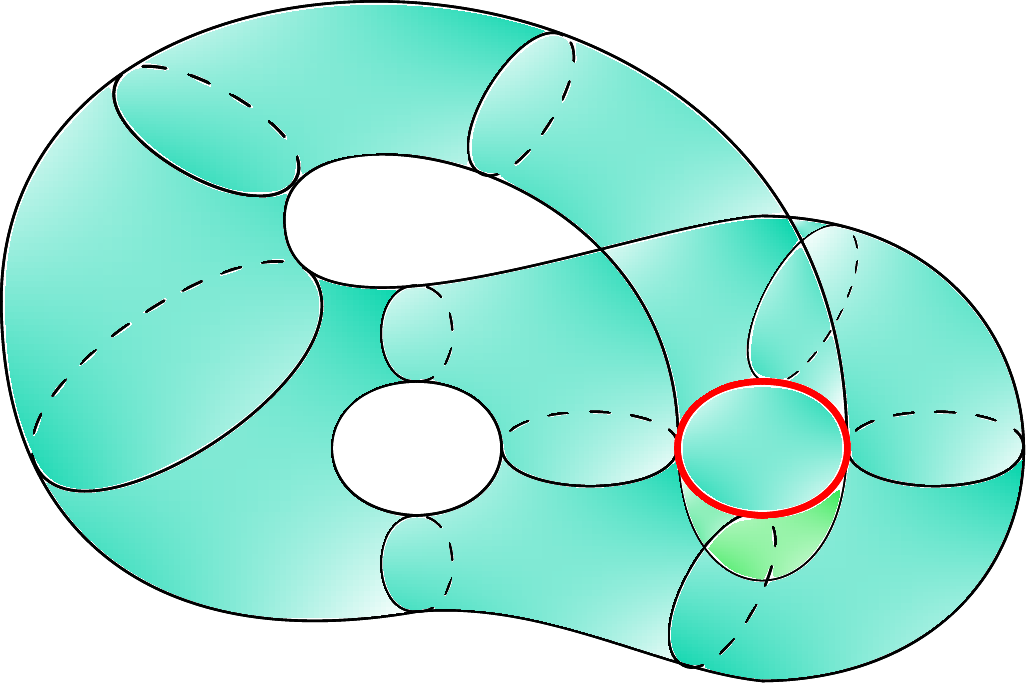}
    \end{overpic}
  \end{tikzcd}
  
\caption{Double loop identification of non-cobordant loops}
\end{figure}

    \item[ii) ] \textbf{$\pmb{\alpha}$ and $\pmb{\beta}$ are cobordant}:
    
Since $\alpha$ and $\beta$ are cobordant, there exists a connected subsurface  $N \subset M_{g}$ such that $\partial(N) = \alpha \cup \beta$. Hence  $M_{g} = N \cup N^{c}$, where $N^{c} \subset M_{g}$ is the \textcolor{black}{closure of}  \textcolor{black}{$M_g - N$}, with $\partial(N^{c}) = \alpha \cup \beta$. Note that $N^{c}$ need not be connected and $N \cap N^{c} = \alpha \cup \beta$.
\textcolor{black}{Also note that the quotient spaces $ N / \alpha \sim \beta$ and $ N^c / \alpha \sim \beta$  are surfaces (since each curve is single-sided in $N$ and $N^c$).
}

Thus, a double loop identification of $\alpha$ and $\beta$ on $M_{g}$ is equivalent to considering the surfaces  $(N / \alpha \sim \beta)$ and $(N^{c} / \alpha \sim \beta)$ intersecting along $\alpha \sim \beta$. It follows that: $$(N / \alpha \sim \beta) = M_{k} \ \ \ \ \ \mbox{ and } \ \ \ \ \ (N^{c} / \alpha \sim \beta) = M_{g - k + 1},$$ for some $k \in \{1, \ldots, g + 1\}$, with $M_{k} \cap M_{g - k + 1}$ homeomorphic to $S^{1}$.

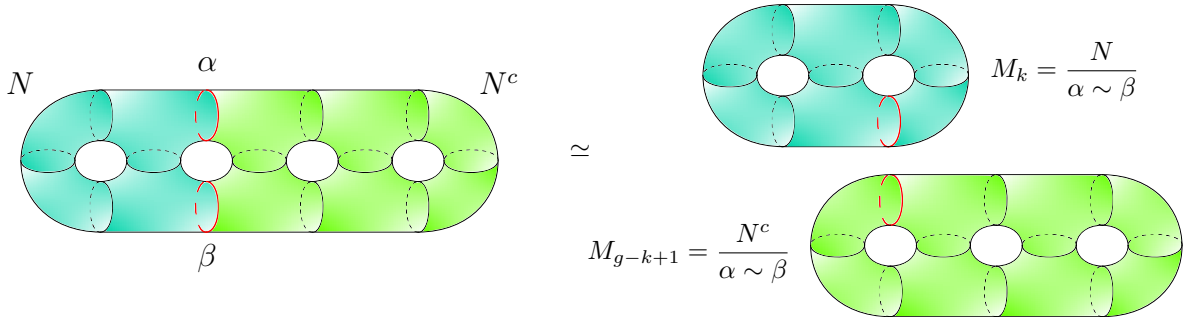
\begin{figure}[h!]
	\begin{minipage}[c]{0.45\linewidth}
		 \begin{equation*}
            \begin{tikzcd}
                \begin{overpic}[scale=0.4]{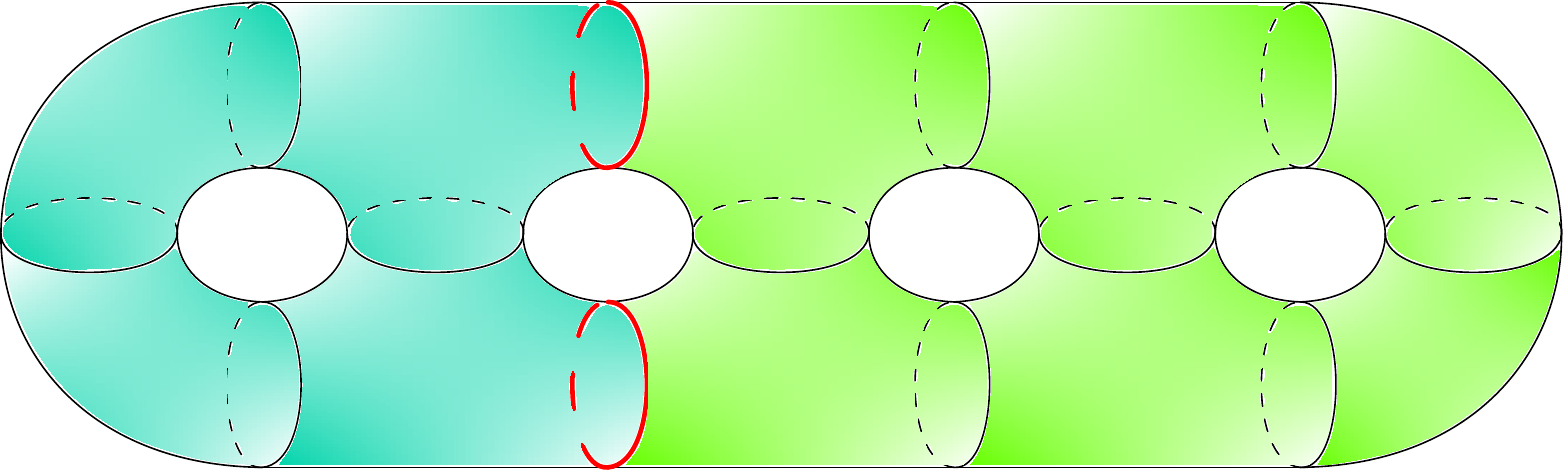}
		            \put(37,34){\large{$\alpha$}}
		            \put(37,-7){\large{$\beta$}}
		            \put(-3,29){\large{$N$}}
		            \put(96,29){\large{$N^{c}$}}
	            \end{overpic}  
            \end{tikzcd}
        \end{equation*}
	\end{minipage}
	\hfill
	\begin{minipage}[c]{0.05\linewidth}
	    $\simeq$
	\end{minipage}
	\hfill
	\begin{minipage}[c]{0.45\linewidth}
	\begin{equation*}
            \begin{tikzcd}
                \begin{overpic}[scale=0.4]{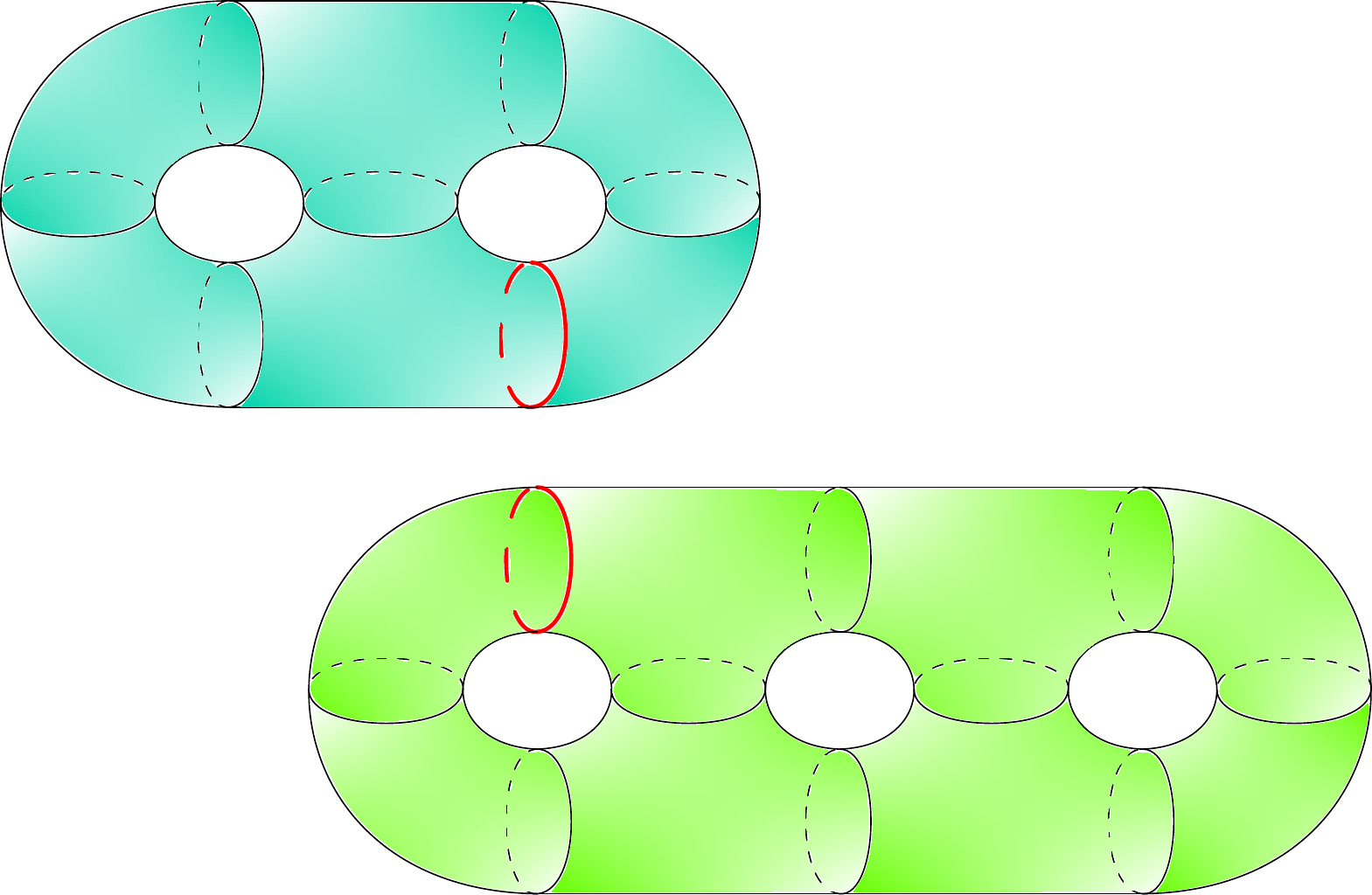}
		            \put(60,50){\normalsize{$M_{k} = \displaystyle\frac{N}{\alpha \sim \beta}$}}
		            \put(-24,12){\normalsize{$M_{g - k + 1} = \displaystyle\frac{N^{c}}{\alpha \sim \beta}$}}
	            \end{overpic}
            \end{tikzcd}
        \end{equation*}
	\end{minipage}
	\caption{Double loop identification of cobordant loops}
	\label{fig:nivaldo}
\end{figure}

Thus, the result follows from the inclusion-exclusion principle for the Euler characteristic: 
$$\mathcal{X}(S(M_{g})) \ = \ \mathcal{X}(M_{k}) + \mathcal{X}(M_{g - k + 1}) - \mathcal{X}(S^{1}) \ = \ (2 - 2k) \ + \ (2 - 2(g - k + 1) )\ = \ 2 - 2g \ = \ \mathcal{X}(M_{g})$$

\end{itemize}

\end{proof}

Now, by using Lemmas~\ref{lem:EulerCone}, \ref{lem:duplo} and \ref{lem:duplo2} the proof of Theorem~\ref{teo:forte} follows.
\begin{proof}[Theorem~\ref{teo:forte}]

By the invariance of the Euler characteristic for double loop identifications proven in Lemmas~\ref{lem:duplo} and \ref{lem:duplo2}, and by applying the inclusion-exclusion principle it follows that: $$ \mathcal{X}(S(\sqcup_{i = 1}^{n}{M_{g_{i}}})) \ = \ \sum_{i=1}^{n}{\mathcal{X}(S(M_{g_{i}}))},$$ where $S(M_{g_{i}})$ is the singular surface obtained from the collection of $M_{g_{i}}$ by performing the collapsing and zipping operations on the  loops in $\mathcal{L}(M_{g_{i}})$. 

According to Lemma~\ref{lem:EulerCone}, the equality holds: $$\mathcal{X}(S(M_{g_{i}})) \ = \ 2 - 2g_{i} + C_{i} + Z_{i},$$ where $C_{i}$ and $Z_{i}$ are the numbers of   loop collapsing operations and loop zipping  operations, respectively, that are performed on $M_{g_{i}}$.

Thus, by adding $\mathcal{X}(M_{g_{i}})$ for $i \in \{1, \ldots, n\}$, the proof follows: 

$$ \mathcal{X}(S(\sqcup_{i = 1}^{n}{M_{g_{i}}})) \ = \ \sum_{i=1}^{n}{\mathcal{X}(S(M_{g_{i}}))} \ = \ \sum_{i = 1}^{n}{2 - 2g_{i} + C_{i} + Z_{i}} \ = \ 2n - 2G + C + Z$$
\end{proof}


In Table~\ref{table-one}  we compute according to Theorem~\ref{teo:forte} the Euler characteristics of the singular surfaces presented as examples in this article.

\begin{table}[!htb]
        \centering
        \resizebox{0.7\linewidth}{!}{
        \begin{tabular}{ccccc}
        
            \cellcolor{gray!20} Singular surface & \cellcolor{gray!20} Smooth data & \cellcolor{gray!20} Collapsings & \cellcolor{gray!20} Zippings & \cellcolor{gray!20} Euler characteristic \\
            \cellcolor{gray!20} & \cellcolor{gray!20} $(n,G)$ & \cellcolor{gray!20} $(C)$ & \cellcolor{gray!20} $(Z)$ & \cellcolor{gray!20} $(\mathcal{X} = 2n - 2G + C + Z)$\\
            \hline
            & & & & \\
            \includegraphics[align=c,scale=0.16]{F8.png} & $(1,0)$ & $1$ & $0$ & $\mathcal{X} = 3$ \\
            & & & &\\
            \hline
            & & & & \\
            \includegraphics[align=c,scale=0.25]{HT2.png} & $(1,1)$ & $1$ & $0$ & $\mathcal{X} = 1$ \\
            & & & & \\
            \hline
            & & & & \\
            \includegraphics[align=c,scale=0.095]{torus_chain.png} & $(5,5)$ & $0$ & $0$ & $\mathcal{X} = 0$ \\
            & & & & \\
            \hline
             & & & & \\
            \hspace{0.6cm} \includegraphics[align=c,scale=0.2]{first2.png} \hspace{0.6cm} & $(1,1)$ & $0$ & $1$ & $\mathcal{X} = 1$ \\
            & & & & \\
            \hline
            & & & & \\
            \includegraphics[align=c,scale=0.25]{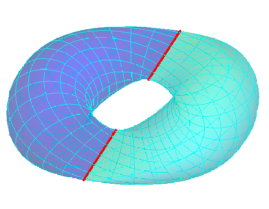} & $(1,1)$ & $0$ & $2$ & $\mathcal{X} = 2$ \\
            & & & & \\
            \hline
            & & & & \\
            \includegraphics[align=c,scale=0.25]{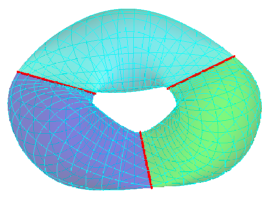} & $(1,1)$ & $0$ & $3$ & $\mathcal{X} = 3$ \\
            & & & & \\
            \hline
            & & & & \\
            \includegraphics[align=c,scale=0.25]{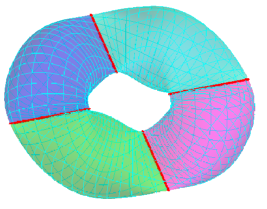} & $(1,1)$ & $0$ & $4$ & $\mathcal{X} = 4$ \\
            & & & & \\
            \end{tabular}}
            \caption{Computation of the Euler characteristic}
            \label{table-one}
    \end{table}

\paragraph{Example.} In Figure \ref{fig:smooth}, consider a collection of three spheres $\{M_0^1, M_0^2, M_0^3\} $ and on it a collection  of loops $$\mathcal{L}(\sqcup_{i=1}^{3}{M_0^i})=\{ (M_0^1,\ell_1,C), (M_0^1,\ell_2; M_0^2,\ell_3,D), (M_0^3,\ell_4,C), (M_0^3,\ell_5,C), (M_0^3,\ell_6,C)(M_0^3,\ell_7,C), (M_0^3,\ell_8,Z)\}. $$ After all loop operations are performed, the singularized manifold  $(S(\sqcup_{i = 1}^{3}{M_{g_{i}}}))$ has two connected components and the Euler characteristic on each connected component is:  $$\mathcal{X}(S(\sqcup_{i = 1}^{2}{M_0^i})) = 2\times2 - 2\times(0) + 1 + 0 = 5 $$ $$\mathcal{X}(S(M_{0}^3)) = 2\times1 - 2\times0 + 4 + 1 = 7$$

\begin{figure}[htb] 
\centering
\begin{equation*}
  \begin{tikzcd}
   \begin{overpic}[scale=0.6]{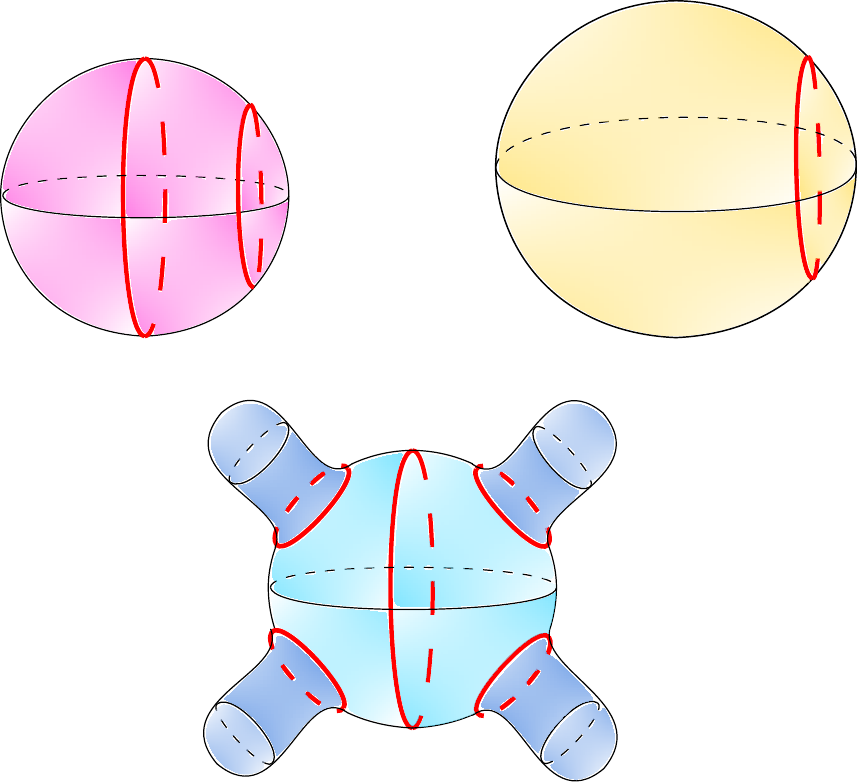}
		\put(15,88){$\ell_{1}$}
		\put(27.5,83){$\ell_{2}$}
		\put(92.5,89){$\ell_{3}$}
		\put(37,40){$\ell_{4}$}
		\put(54,40){$\ell_{5}$}
		\put(25,16){$\ell_{6}$}
		\put(66.5,16){$\ell_{7}$}
		\put(47,1){$\ell_{8}$}	
		\put(-20,68){$M_{0}^{1}$}
		\put(110,70){$M_{0}^{2}$}
		\put(0,20){$M_{0}^{3}$}
	\end{overpic} \arrow{rrr}{Singularization}  & & & \hspace{2cm} \begin{overpic}[scale=0.6]{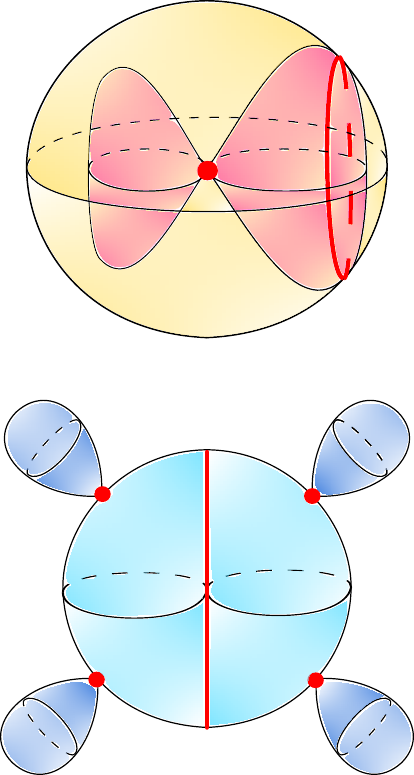} 
	    \put(55,77){$S(\sqcup_{i = 1}^{2}{M_0^i})$}
	    \put(62,22){$S(M_{0}^3)$}
	\end{overpic}
  \end{tikzcd}
\end{equation*}
\caption{Example of singularization of three spheres}
\label{fig:smooth}
\end{figure}
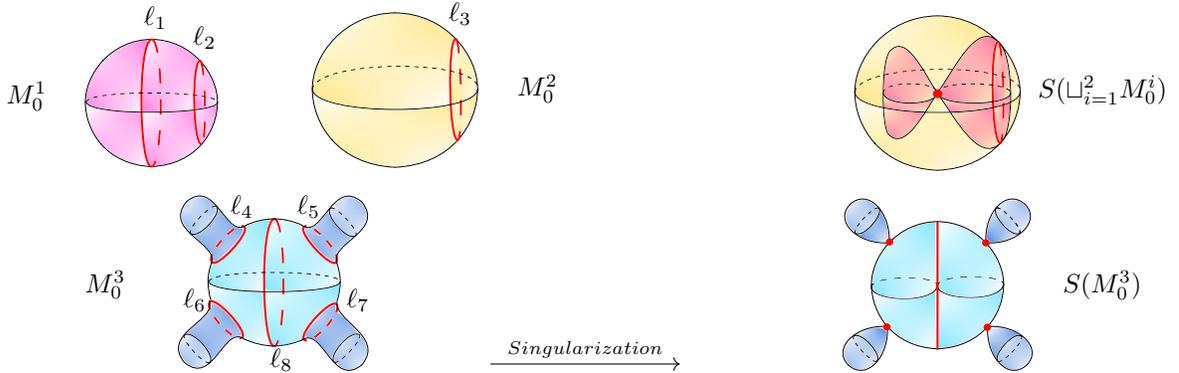

\begin{Def}
The \textbf{genus} $g^{S}$ of a   singularized connected surface  $S(\sqcup_{i=1}^{n}M_{g_{i}})$ is the maximal number of \textcolor{black}{disjoint} simple closed curves that can be removed from its  \textcolor{black}{nonsingular}  part without disconnecting $S(\sqcup_{i=1}^{n}M_{g_{i}})$.
\end{Def}

Note that the above definition generalizes the classical definition of genus for a smooth surface. Furthermore, the restriction on the removal of simple closed curves to the \textcolor{black}{nonsingular} part of the singularized surface, avoids the problem of infinitely many simple closed curves intersecting a one dimensional singular set  without disconnecting it.

The next lemma shows that the genus of a connected singularized surface $S(\sqcup_{i=1}^{n}M_{g_{i}})$ depends on the genus of the surfaces $M_{g_i}$ and the number of  double loop identifications performed. Moreover, it is invariant with  respect to    \textcolor{black}{collapsing and  zipping}.

\begin{Lem}\label{lem:genus}
Let $S(\sqcup_{i=1}^{n}M_{g_{i}})$ be a connected singularized surface obtained from the collection  $\{M_{g_{i}} \ | \ i = 1, \ldots n\}$ of smooth surfaces by a singularization process. Then $$g^{S} = \left(\sum_{i=1}^{n}{g_{i}}\right) + [D - (n - 1)],$$ where $D$ is the number of double loop identification in $\mathcal{L}(\sqcup_{i=1}^{n}M_{g_{i}})$.
\end{Lem}

\begin{proof}

Notice that, in order for a singularized surface $S(\sqcup_{i=1}^{n}M_{g_{i}})$ to be connected, the minimum number $D$ of double loop identifications in $\mathcal{L}(\sqcup_{i=1}^{n}M_{g_{i}})$ needed to achieve this is $n - 1$. Thus, define $k = D - (n - 1)$ as the  \textcolor{black}{\bf number of exceeding double loop identifications in the singularization}. The proof will follow by induction on $k$.
\begin{itemize}
    \item[a) ] First, suppose $k = 0$, that is, $D = n - 1$;
    
     \textcolor{black}{In the nonsingular} case, there is a  \textcolor{black}{one-to-one} correspondence between the genus $g_{i}$ and the number of handles  on the surface $M_{g_{i}}$.  \textcolor{black}{So} for each handle, one can pick either one of its handle loops or one of its tunnel loops to be removed, and  \textcolor{black}{the number of handles} gives the maximum number $g_{i}$ of  \textcolor{black}{disjoint} simple closed curves that can be removed from $M_{g_{i}}$ without disconnecting it. 
    
    In the singular case, one can proceed in a similar fashion, since the genus, $g^{S}$,  is equivalent to the maximum number of simple closed curves in $\sqcup_{i=1}^{n}M_{g_{i}}$, not intersecting loops in $\mathcal{L}(\sqcup_{i=1}^{n}M_{g_{i}})$, which can be removed from $S(\sqcup_{i=1}^{n}M_{g_{i}})$ without disconnecting it. 
    Given a handle in $\sqcup_{i = 1}^{n}M_{g_{i}}$, if one of its handle (resp. tunnel) loops is in $\mathcal{L}(\sqcup_{i=1}^{n}M_{g_{i}})$, one chooses a handle (resp. tunnel) loop to be removed. Otherwise, one can choose between removing a handle or a tunnel loop.
    
   Proceeding as above for each  handle  in $\sqcup_{i=1}^{n}M_{g_{i}}$, one removes a total of $\sum_{i=1}^{n}g_{i}$ simple closed curves in the \textcolor{black}{nonsingular} part of $S(\sqcup_{i=1}^{n}M_{g_{i}})$ without disconnecting it. Thus, it follows that $g^{S} \geq \sum_{i=1}^{n}g_{i}$.
    
    Suppose by contradiction that $g^{S} > \sum_{i=1}^{n}g_{i}$. Then there is one more simple closed curve $\alpha$ in the  \textcolor{black}{nonsingular} part of $S(\sqcup_{i=1}^{n}M_{g_{i}})$ that can be removed without disconnecting it. However, $\alpha$ must lie in $M_{g_{i}}$ for some $i = 1, \ldots, n$. This means that $g_{i} + 1$ simple closed curves are being removed from $M_{g_{i}}$. Hence $M_{g_{i}}$ becomes \textcolor{black}{disconnected.  Now} we have $n + 1$ disjoint connected components in $\sqcup_{i=1}^{n}M_{g_{i}}$ and only $D = n - 1$ double loop operations in $\mathcal{L}(\sqcup_{i=1}^{n}M_{g_{i}})$; \textcolor{black}{this} is clearly not enough to obtain a connected space. Thus $S(\sqcup_{i=1}^{n}M_{g_{i}})$ is not connected after the removal of $\alpha$ which is a contradiction.
    
    Therefore, $g^{S} = \sum_{i=1}^{n}g_{i}$.
    
    \item[b) ] Suppose \textcolor{black}{$k\geq 1$ and}  the formula holds in the case that there are $k - 1$ exceeding double loop identifications in $\mathcal{L}(\sqcup_{i=1}^{n}M_{g_{i}})$. That is, $$g^{S} = \sum_{i=1}^{n}{g_{i}} + [k - 1].$$
    
 We prove that the formula still holds if a $k$-th exceeding double loop identification is performed.
   
  Suppose $\mathcal{L}(\sqcup_{i=1}^{n}M_{g_{i}})$ contains $k$ exceeding double loop identifications and determines the singularized surface $S(\sqcup_{i=1}^{n}M_{g_{i}})$. 
   Let $S^\prime(\sqcup_{i=1}^{n}M_{g_{i}})$ be the singularized surface determined by $\mathcal{L}(\sqcup_{i=1}^{n}M_{g_{i}})$ with $k-1$ exceeding double loop identifications performed, leaving out the pair of disjoint  loops $(\alpha, \alpha^\prime)$ that will eventually be doubly identified.
    By the induction hypothesis, a total of $g^{S^\prime}=\sum_{i=1}^{n}{g_{i}} + [k - 1]$  \textcolor{black}{disjoint} simple closed curves can be removed from the \textcolor{black}{nonsingular} part of $S^\prime(\sqcup_{i=1}^{n}M_{g_{i}})$ without disconnecting it.        The removal of any other \textcolor{black}{simple closed} curve $\gamma$ in the \textcolor{black}{nonsingular} part of $S^\prime(\sqcup_{i=1}^{n}M_{g_{i}})$ will divide it into two disjoint connected components. One can choose  $\gamma$ in such a way that  these disjoint connected components $S_1^\prime$ and $S_2^\prime$ contains  the loops $\alpha$ and $ \alpha^\prime$ respectively. Therefore, by performing the double loop identification on $(\alpha, \alpha^\prime)$, these two components $S_1^\prime$ and $S_2^\prime$ become connected forming $S(\sqcup_{i=1}^{n}M_{g_{i}})$. Hence, we have shown that $S(\sqcup_{i=1}^{n}M_{g_{i}})$ remains connected after the removal of $\gamma$, meaning that:  $$g^{S} = \sum_{i=1}^{n}{g_{i}} + [k - 1] + 1 = \left(\sum_{i=1}^{n}{g_{i}}\right) + k = \left(\sum_{i=1}^{n}{g_{i}}\right) + [D - (n - 1)],$$ concluding the proof.

\end{itemize}

\end{proof}

\paragraph{Example.} In Figure \ref{fig:genus}, a connected singularized surface is obtained from the disjoint union of smooth surfaces, specifically a sphere $M_{0}$, a torus $M_{1}$, and a $3$-torus $M_{3}$ by performing the operations in $$\mathcal{L}(M_{0}\sqcup M_{1} \sqcup M_{3}) = \{ (M_{1}, \ell_{1}; M_{3}, \ell_{4}, D), (M_{1}, \ell_{2}; M_{3}, \ell_{6}, D), (M_{0}, \ell_{3}; M_{3}, \ell_{5}, D)  \}.$$
By Lemma \ref{lem:genus},  $g^{S} = 1 + 0 + 3 +[ 3 - (3 - 1)] = 5$. \textcolor{black}{In Figure  \ref{fig:genus}, we show five disjoint curves missing the singularities that don't separate the singularized surface.}
\begin{figure}[!h]
    \begin{equation*}
            \begin{tikzcd}
                \begin{overpic}[scale=0.5]{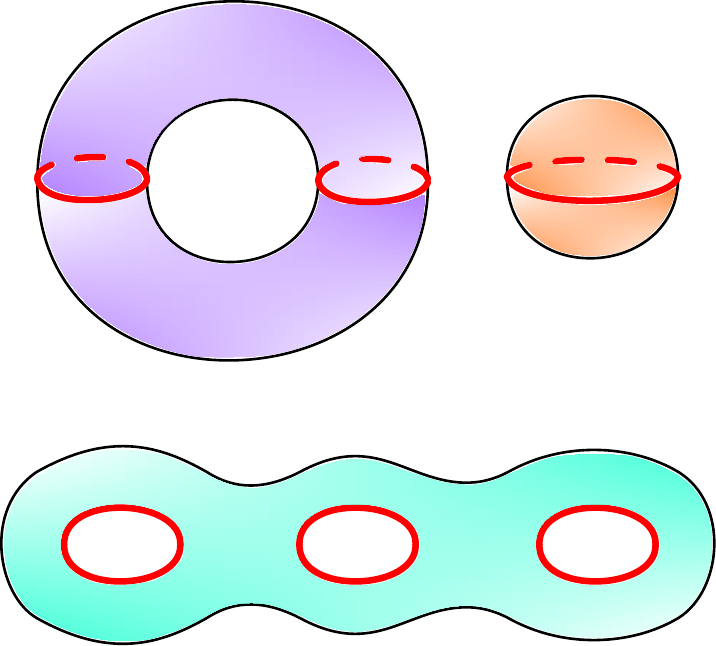}
		\put(-5,62){$\ell_{1}$}
		\put(34,62){$\ell_{2}$}
		\put(98,62){$\ell_{3}$}
		\put(13,11.5){$\ell_{4}$}
		\put(46,11.5){$\ell_{5}$}
		\put(79,11.5){$\ell_{6}$}
		\put(77,90){$M_{0}$}
		\put(27,105){$M_{1}$}
		\put(43,-16){$M_{3}$}
	\end{overpic} \hspace{0.8cm}  \xrightarrow[identification]{Double \ loop}  \includegraphics[scale=0.5]{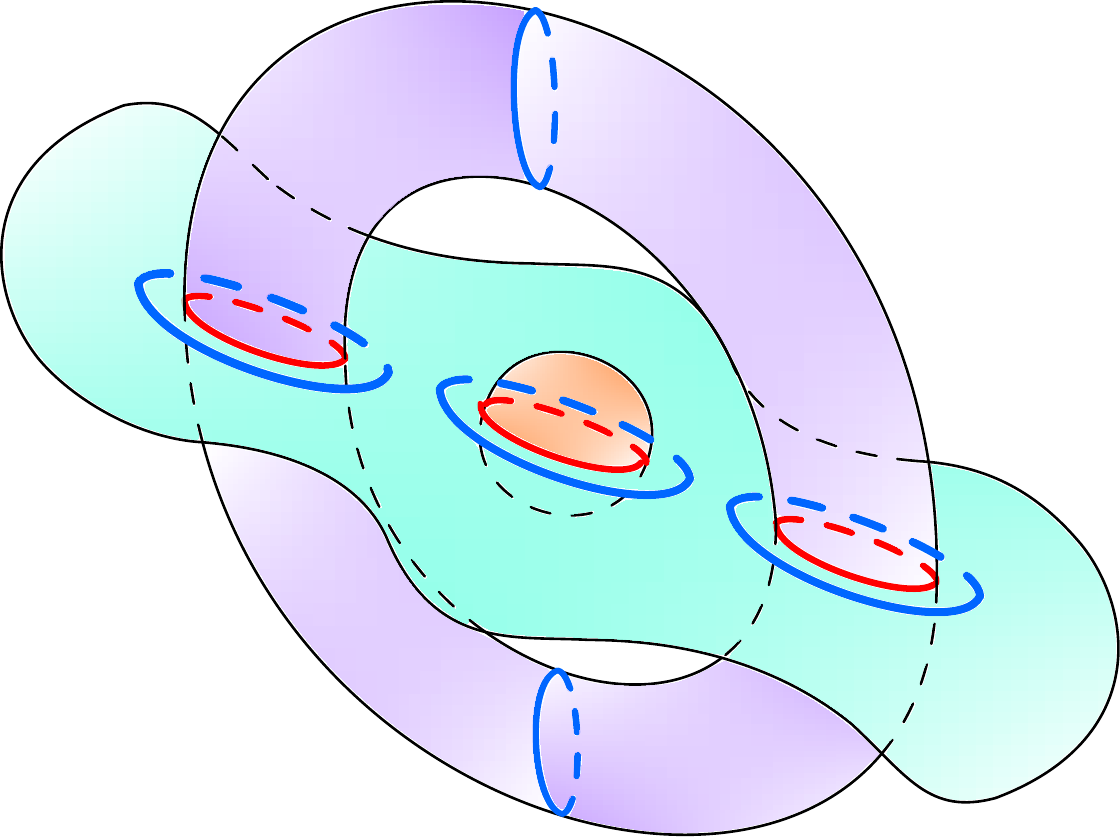} 
            \end{tikzcd}
        \end{equation*}
		\caption{Genus $5$ singularized surface}
		\label{fig:genus}
\end{figure}
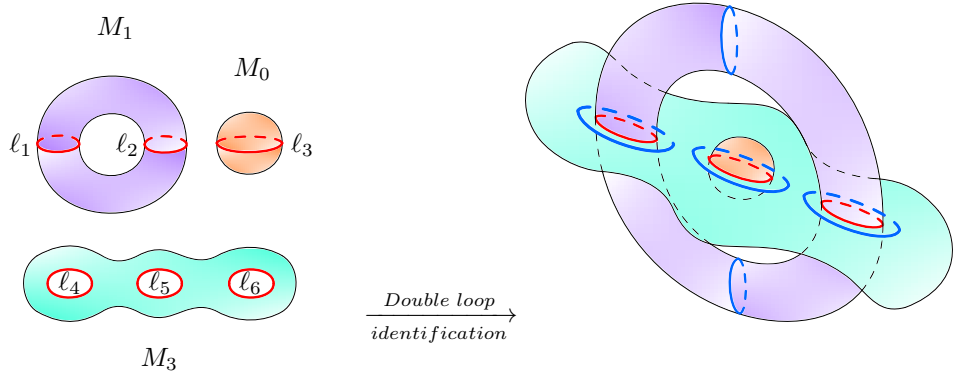

\begin{Teo}
Let $S(\sqcup_{i=1}^{n}M_{g_{i}})$ be a connected singularized surface obtained from the disjoint union of smooth surfaces $\{M_{g_{i}} \ | \ i = 1, \ldots n\}$ by the singularization process determined by $\mathcal{L}(\sqcup_{i=1}^{n}M_{g_{i}})$.  Then the Euler characteristic of  $S(\sqcup_{i=1}^{n}M_{g_{i}})$ is given by: $$\mathcal{X}(S(\sqcup_{i=1}^{n}M_{g_{i}})) = 2 - 2g^{S} + 2D + C + Z,$$ where $C$, $Z$ and $D$ are, respectively, the number of   \textcolor{black}{collapses,  zips}  and double loop identifications in $\mathcal{L}(\sqcup_{i=1}^{n}M_{g_{i}})$.
\end{Teo}

\begin{proof}
By Theorem \ref{teo:forte}, we have that:
\begin{equation}\label{eq:teo}
    \mathcal{X}(S(\sqcup_{i=1}^{n}M_{g_{i}})) = 2n - 2G + C + Z
\end{equation}
where $G=\sum_{i=1}^{n}{g_{i}}$.

By Lemma \ref{lem:genus}, 
\begin{equation}\label{eq:genus}
    G = g^{S} - D + (n - 1)
\end{equation}

Hence, the result follows by substituting (\ref{eq:genus}) in (\ref{eq:teo}):
\begin{eqnarray*}
    \mathcal{X}(S(\sqcup_{i=1}^{n}M_{g_{i}})) & = & 2n - 2[g^{S} - D + (n - 1)] + C + Z \\
    & = & 2 - 2g^{S} + 2D + C + Z.
\end{eqnarray*}

\end{proof}

We conclude this article by remarking that many interesting questions arise in the context of singularization. For instance, one can explore the dependency of the singularized surface to the loop operations assigned in a singularization. When are different assignments of loop operations topologically equivalent? What are the ranges of the Euler characteristic attainable in this case? 

Note that, in general,  the Euler characteristic is not a complete topological invariant for singularized surfaces, two singularized surfaces may have the same Euler characteristic but not be homeomorphic.
 Indeed, it is easy to find two singularized surfaces having the same Euler characteristic but with different singular sets. Are the singularized surfaces given here classifiable? 

More complex singularization operations can be investigated. A $3$-sheet cone and a triple crossing are examples of more degenerate singular sets that appear to be attainable by quotient maps similar to the collapsing and the double loop identification, respectively. Are all the singular sets produced by quotient maps of loops more degenerated cases of the singularities discussed here?

Also one can't help but wonder the effect on the Euler Characteristic of simple loop operations on a closed surface S where the family $L$ of loops are not necessarily disjoint. 

The explicit computation of the Betti numbers of a singularized surface can be posed as well, since their alternating sum yields another proof for the Euler characteristic formula presented here.
Moreover, one can search if there is a relation between the Betti numbers of  a singularized surface and its genus, since  in the smooth case the genus of a surface is half its first Betti number.

There are many interesting and accessible questions that can be taken up from where this article left off. We entrust our reader will accept the challenge.

\end{document}